\documentclass[11pt, a4paper]{amsart}
\usepackage[utf8]{inputenc}

\usepackage{amsmath, amsfonts, amssymb, amsthm,mathtools,enumerate}
\allowdisplaybreaks
\usepackage{graphicx}
\usepackage{color}
\usepackage[inline]{enumitem}
\usepackage{hyperref}
\usepackage{multicol}
\usepackage{tikz-cd}
\usepackage[margin=1in]{geometry}
\usepackage{amsmath,amssymb,amsthm,mathtools,mathrsfs}
\usepackage[T1]{fontenc}
\usepackage{lmodern}
\usepackage{microtype}
\usepackage{shuffle}

\usepackage{ulem}
\usepackage{soul}
\newcommand{\Add}[1]{\textcolor{red}{#1}} 

\newtheorem{thm}{Theorem}[section]

\newtheorem{prop}[thm]{Proposition}

{\theoremstyle{definition} \newtheorem{rem}[thm]{Remark}}
{\theoremstyle{definition} \newtheorem{defn}[thm]{Definition}}
{\theoremstyle{definition} \newtheorem{thm-defn}[thm]{Theorem-Definition}}
{\theoremstyle{definition}  \newtheorem{eg}[thm]{Example}}
{\theoremstyle{definition}  }

{\theoremstyle{definition} }
{\theoremstyle{definition} }

\numberwithin{equation}{section}

\newcommand{\Li}{\operatorname{Li}}

\providecommand{\Li}{\operatorname{Li}}

\providecommand{\reg}{\operatorname{reg}}

\title[Iterated cotangent-difference integral]
{Iterated cotangent-difference integrals \\
over totally real cyclotomic fields}

\author{Jin Cao}
\address{School of Mathematics and Physics, University of technology and science Beijing, Beijing, China}
\email{caojin@ustb.edu.cn}

\author{Hidekazu Furusho}
\address{Graduate School of Mathematics, Nagoya University, Furo-cho, Chikusa-ku, Nagoya, 464-8602,
Japan}
\email{furusho@math.nagoya-u.ac.jp}

\date{October 2, 2026}

\begin{document}

\begin{abstract}
%
We study iterated cotangent integrals with cyclotomic shifts and express their values along specified reference paths in terms of cyclotomic multiple zeta values. Differences of cotangent forms extend to a punctured projective line over a totally real cyclotomic field, giving a family of ``real cyclotomic multiple zeta values.'' With rational tangential endpoints, these difference integrals are periods of mixed Tate motives over that field; integral endpoint data yield periods over its ring of integers after inverting primes dividing the level. The associated framed motivic classes generate a Hopf subalgebra under Goncharov's coproduct.
\end{abstract}


\maketitle
\setcounter{tocdepth}{2}
\tableofcontents
\setcounter{section}{-1}
\section{Introduction}
\label{sec:introduction}
Cyclotomic multiple zeta values arise as iterated integrals on
$\mathbb{P}^1\setminus\{0,\mu_N,\infty\}$ and carry rich arithmetic
structures. Their study draws on Chen's theory of iterated integrals
\cite{Chen1977} and the work of Goncharov,
Borwein--Bradley--Broadhurst--Lisonek, and Racinet
\cite{Goncharov1998,BorweinEtAl2001,Racinet2002}.
The motivic and Galois-theoretic aspects were developed in
\cite{DeligneGoncharov2005,Goncharov2005,Brown2012};
in particular, Glanois studied Galois descent for cyclotomic
motivic iterated integrals \cite{Glanois2016}.

Cyclotomic Multiple zeta values have also been studied through trigonometric
integrals, including iterated log-sine and log-tangent integrals
by Umezawa
\cite{Umezawa2020, Umezawa2023,UmezawaLogTangent}.
Bachmann~\cite{BachmannShiftedLogSine2026} proved that Umezawa's
shifted log-sine integrals with odd entries are ordinary multiple
zeta values. He also interpreted these values using iterated
integrals in \(d\theta\) and \(\cot\theta\,d\theta\) on
\([0,\pi/6]\).
In this paper, we study shifted cotangent forms and their
differences, with particular emphasis on their geometry over
totally real cyclotomic fields.

In this paper, we start with iterated cotangent integrals.
Taking differences of the underlying cotangent forms yields
iterated integrals on a punctured projective line over a
totally real cyclotomic field.
Fix \(N\geq3\), let \(E=\mathbb Q(\mu_N)\) and \(F=E^+\).
Let
$S$ be the set of finite places of $F$ lying above the primes dividing
$N$.
For \(a\in\mathbb Z/N\mathbb Z\),
put \(\theta_a=\pi a/N\) on \(\mathbb C/\pi\mathbb Z\). We study the \textit{iterated cotangent integrals}
of the following type
\[
 C_\gamma^{\mathrm{reg}}(a_0;\,a_1,\ldots,a_r;\,b_0)
 :=\operatorname{reg}\int_\gamma
 \eta_{a_1}^{\cot}\cdots\eta_{a_r}^{\cot},
 \qquad
 \eta_{a_i}^{\cot}:=\cot(\theta-\theta_{a_i})\,d\theta,
\]
for \( a_0,b_0,a_1,\ldots,a_r\in\mathbb Z/N\mathbb Z\)
where \(\gamma:\theta_{a_0}\rightsquigarrow\theta_{b_0}\) avoids the punctures in its interior. The path includes specified tangential endpoint data, with regularization as in \cite{Deligne1989}. 
These integrals iterate the differentials \(d\log\sin(\theta-\theta_a)\).

The forms \(\eta_a^{\cot}\) have algebraic representatives on the
\textit{cotangent punctured line} \(X_N^{\cot}\) over \(F\).
Its boundary consists of \(F\)-rational points and one nonsplit quadratic point. After base change to \(E\), a Cayley transformation identifies it with the cyclotomic punctured line. 

Taking differences of integrands removes the contribution of the
quadratic boundary.  More precisely, put
\[
 \eta_{a,b}^{\mathrm{diff}}
 :=
 \eta_a^{\cot}-\eta_b^{\cot}.
\]
For pairs \((a_0,b_0),(a_1,b_1),\ldots,(a_r,b_r) \in \mathbb{Z}/N\mathbb{Z} \times \mathbb{Z}/N\mathbb{Z}\), we define the corresponding
\textit{iterated cotangent-difference integral} by
\[
\begin{aligned}
 C_\gamma^{\mathrm{diff},\mathrm{reg}}
 \bigl(
  a_0;\,(a_1,b_1),\ldots,(a_r,b_r);\,b_0
 \bigr)
 &:=
 \operatorname{reg}\int_\gamma
 \eta_{a_1,b_1}^{\mathrm{diff}}\cdots
 \eta_{a_r,b_r}^{\mathrm{diff}}.
\end{aligned}
\]
The algebraic representatives $\kappa_{a,b}^{\mathrm{diff}}$ of these difference forms $ \eta_{a,b}^{\mathrm{diff}}$ extend across the nonsplit quadratic point to logarithmic forms on the \textit{difference punctured line} \(X_N^{\mathrm{diff}}\), whose punctures are all \(F\)-rational; see Definition~\ref{def:cotangent-difference-integral}.  Our main results are as follows.



\begin{enumerate}
\item \emph{Cyclotomic expressions.}
For the reference paths \(\gamma_{m,n}\) with their prescribed tangential data, Theorem~\ref{thm:cyclotomic-endpoint-formula} gives an explicit formula in radial iterated integrals and tangential rotations. In particular,
\[
 C_{\gamma_{m,n}}^{\mathrm{reg}}(a_0;\,a_1,\ldots,a_r;\,b_0)
 \in\mathcal Z_{N,r},
\]
where \(\mathcal Z_{N,r}\) denotes the \(\mathbb Q\)-span of regularized level-\(N\) cyclotomic multiple zeta values of weight \(r\).

\item \emph{Mixed Tate descent.}
The difference punctured line \(X_N^{\mathrm{diff}}\)
has good reduction over \(\mathcal O_F[1/N]\). 
By the Deligne--Goncharov construction, the iterated
cotangent-difference integral along a Betti path
\(\gamma_v\) is a comparison period of a mixed Tate motive over \(F\),
provided its endpoint tangential base points are defined over \(F\).
Theorem~\ref{thm:cotangent-difference-mixed-tate} asserts that,
if the endpoints are \(S\)-integral, that is, if they generate the tangent lines of the chosen model, then the integral is a period over
\(\mathcal O_{F,S}\), that is, 
\[
 C_{\gamma}^{\mathrm{diff},\mathrm{reg}}
 \bigl(a_0;\,(a_1,b_1),\ldots,(a_r,b_r);\,b_0\bigr)
 \in
 \operatorname{Per}_\sigma
 \bigl(\operatorname{MT}(\mathcal O_F[1/N])\bigr).
\]
Here \(\operatorname{Per}_\sigma\) denotes the algebra of complex
comparison periods for the chosen embedding
\(\sigma:F\hookrightarrow\mathbb C\).

\item \emph{Hopf algebra structure and Goncharov's coproduct.}
Let \(\mathcal B_N\) be the fixed system of integral tangential
base points at all punctures; see
\eqref{Definition of basepoint set}.
For \(\beta_0,\beta_1\in\mathcal B_N\), the motivic path torsor
between these base points gives rise to both comparison periods
and Tate-framed classes. Pairing a word in the difference forms
with a Betti path gives the numerical integral
\[
 C_\gamma^{\mathrm{diff},\mathrm{reg}}
 \bigl(a_0;\,(a_1,b_1),\ldots,(a_r,b_r);\,b_0\bigr),
\]
where \(\beta_0,\beta_1\) are the transported tangential
endpoints. 
The same motivic path torsor, equipped with rational Tate
framings, defines the path-independent class
\[
 I^{\mathfrak m}
 \bigl(\beta_0;
 \kappa^{\mathrm{diff}}_{a_1,b_1}|\cdots|
 \kappa^{\mathrm{diff}}_{a_r,b_r};\beta_1\bigr)
 \in\mathcal A^{\mathcal M}_{F,S}.
\]
Section~\ref{sec:cotangent-difference-motivic-coaction}
constructs these classes and the algebra
\(\mathcal H_N^{\mathrm{diff},\mathfrak m}(\mathcal B_N)\)
they generate.
Theorem~\ref{thm:cotangent-difference-motivic-hopf}
gives an explicit Goncharov coproduct formula and proves that
this algebra is a connected graded Hopf subalgebra of
\(\mathcal A^{\mathcal M}_{F,S}\), the coordinate algebra of
the unipotent mixed Tate motivic Galois group over
\(\mathcal O_{F,S}\).


\item \emph{Low-weight examples. } 
Our construction provides a family of periods over real cyclotomic fields. 
In length one, suitable difference integrals give logarithms of ratios of real cyclotomic \(S\)-units; compare \cite{Sinnott1978}. Higher-weight examples at level four recover classical zeta values. See Section \ref{sec:low-weight-examples}.
\end{enumerate}

The paper is organized as follows. Section~\ref{sec:iterated-cotangent-integrals} introduces the integrals and their geometry, and Section~\ref{sec:cyclotomic-realization} gives their cyclotomic expressions. Section~\ref{sec:cotangent-difference-integrals} constructs the difference punctured line and its integral model. Section~\ref{sec:cotangent-difference-motivic-coaction} establishes the mixed Tate interpretation and Hopf algebra structure, and Section~\ref{sec:low-weight-examples} gives explicit evaluations.

\section{Iterated cotangent integrals and the cotangent punctured line}
\label{sec:iterated-cotangent-integrals}
Throughout this section, let \(N\geq3\), and put
\[
 \zeta_N:=\exp\left(\frac{2\pi\sqrt{-1}}{N}\right),
 \qquad
 E:=\mathbb Q(\mu_N) \  \ (\subset {\mathbb C}),
 \qquad
 F:=E^+.
\]
In this section, we first define iterated cotangent integrals analytically on the angular cylinder.  We then realize them as iterated integrals of logarithmic differential forms on the cotangent punctured line, a punctured projective line defined over the totally real cyclotomic field \(F\).

\subsection{Iterated cotangent integrals}
\label{subsec:definition-iterated-cotangent-integrals}

This subsection fixes the analytic definition of iterated cotangent
integrals and the tangential regularization used throughout the paper.

Let
\[
 \Theta:=\mathbb C/\pi\mathbb Z
\]
be the angular cylinder.  For \(a\in\mathbb Z/N\mathbb Z\), put
\[
 \theta_a:=\frac{\pi a}{N}\pmod{\pi},
 \qquad
 e^{2\sqrt{-1}\theta_a}=\zeta_N^a,
\]
and define
\[
 \Theta_N^\circ
 :=
 \Theta\setminus\{\theta_a:a\in\mathbb Z/N\mathbb Z\}.
\]
The cotangent logarithmic form attached to \(a\) is
\[
 \eta_a^{\cot}
 :=
 \cot(\theta-\theta_a)\,d\theta
 =
 d\log\sin(\theta-\theta_a).
\]
It is a logarithmic differential form on \(\Theta_N^\circ\), with a simple
pole at \(\theta_a\).

We use tangential base points at the punctures.  Thus, a tangential base
point at \(\theta_c\) is a nonzero vector
\(
 \vec v_c\in T_{\theta_c}\Theta.
\)
It fixes a local parameter \(t_c\) by
\[
 t_c(\theta_c)=0,
 \qquad
 dt_c(\vec v_c)=1.
\]
For instance, if \(\vec v_c=\nu_c\partial_\theta\) in a local lift of
\(\theta\), we may take
\( t_c=\frac{\theta-\theta_c}{\nu_c}.\)

Let \(\gamma\) be a path with interior in \(\Theta_N^\circ\) and with
tangential base points at its two endpoints.  If
\(\omega_1,\ldots,\omega_r\) are logarithmic forms, we define
\[
 \operatorname{reg}\int_\gamma\omega_1\cdots\omega_r
 :=
 \operatorname{CT}_{\log\varepsilon_0,\log\varepsilon_1}
 \int_{\gamma_{\varepsilon_0,\varepsilon_1}}
 \omega_1\cdots\omega_r.
\]
Here \(\gamma_{\varepsilon_0,\varepsilon_1}\) is obtained by truncating
\(\gamma\) near its initial and final endpoints using the normalized local
parameters, and \(\operatorname{CT}\) denotes the constant term in the
resulting polynomial in
\(\log\varepsilon_0\) and \(\log\varepsilon_1\).  
This is the standard tangential regularization of iterated integrals; see
Deligne~\cite{Deligne1989}.
It agrees with the usual iterated integral
when the latter converges.

\begin{defn}
\label{def:iterated-cotangent-integral}
Let
\( a_0,a_1,\ldots,a_r,b_0\in\mathbb Z/N\mathbb Z,\)
and let
\(\gamma:\theta_{a_0}\rightsquigarrow\theta_{b_0}\)
be a path whose interior lies in \(\Theta_N^\circ\), equipped with
tangential base points at its endpoints.
The associated \textit{iterated
cotangent integral} is
\[
 C_\gamma^{\reg} (a_0;\,a_1,\ldots,a_r;\,b_0)
 :=
 \operatorname{reg}\int_\gamma
 \eta_{a_1}^{\cot}\cdots\eta_{a_r}^{\cot}.
\]
\end{defn}

When the endpoint labels are clear from \(\gamma\), we abbreviate this as
\[
 C_\gamma^{\reg}(a_1,\ldots,a_r)
 :=
 \operatorname{reg}\int_\gamma
 \eta_{a_1}^{\cot}\cdots\eta_{a_r}^{\cot}.
\]
If \(\gamma\) is represented by an oriented real interval with no
intermediate pole, then
\[
\begin{aligned}
 C_\gamma^{\reg}
 (a_0;\,a_1,\ldots,a_r;\,b_0)
 =
 \operatorname{reg}
 \int_{\theta_{a_0}<t_1<\cdots<t_r<\theta_{b_0}}
 \prod_{j=1}^r
 \cot(t_j-\theta_{a_j})\,dt_j.
\end{aligned}
\]

The path \(\gamma\), including its tangential endpoint data, is part of the
datum.  In particular, different relative homotopy classes of paths can
give different iterated cotangent integrals.

\begin{rem}
Bachmann~\cite[Section~4]{BachmannShiftedLogSine2026}
defines iterated cotangent integrals using the forms
\(d\theta\) and \(\cot\theta\,d\theta\) on \([0,\pi/6]\).
In Remark~4.5 of loc.~cit., he also observes that the upper endpoint can
be replaced by \(\pi/n\), subject to endpoint convergence
conditions. Our definition uses the shifted forms
\(\eta_a^{\cot}\), with paths between cyclotomic angles and
tangential regularization. The common form is
\(\eta_0^{\cot}=\cot\theta\,d\theta\); the form \(d\theta\)
is not included among our defining forms.
\end{rem}

The following examples illustrate ordinary convergence, tangential regularization, and the effect of iterating different cotangent kernels. All real intervals below are oriented from left to right, with endpoint tangent vectors \(+\partial_\theta\) at the initial point and \(-\partial_\theta\) at the final point.

\begin{eg}[Length one]
Suppose that \(\gamma\) is represented by a real interval from \(\theta_{a_0}\) to \(\theta_{b_0}\), with no puncture in its interior, and that \(a_1\neq a_0,b_0\). Using real lifts of the angular coordinates, we obtain
\[
 C_\gamma^{\reg}(a_0;\,a_1;\,b_0)
 =
 \int_{\theta_{a_0}}^{\theta_{b_0}}
 \cot(\theta-\theta_{a_1})\,d\theta
 =
 \log\left|
 \frac{\sin(\theta_{b_0}-\theta_{a_1})}
      {\sin(\theta_{a_0}-\theta_{a_1})}
 \right|.
\]
No regularization is needed in this case. For example, at level \(N=4\), the real interval \(\gamma:\pi/4\rightsquigarrow\pi/2\) gives
\[
 C_\gamma^{\reg}(1;\,0;\,2)
 =
 \log\frac{\sin(\pi/2)}{\sin(\pi/4)}
 =
 \frac12\log 2.
\]
\end{eg}

\begin{eg}
Take \(N=4\), and let \(\gamma_{\mathrm{re}}:0\rightsquigarrow\pi/4\) be the increasing real interval with the tangent vectors specified above. The form \(\eta_0^{\cot}=\cot\theta\,d\theta\) has a pole at the initial endpoint. Since its normalized local parameter is \(\theta\), we have
\[
 \int_\varepsilon^{\pi/4}\cot\theta\,d\theta
 =
 \log\sin(\pi/4)-\log\sin\varepsilon
 =
 -\log\varepsilon-\frac12\log 2+O(\varepsilon^2).
\]
Taking the constant term therefore gives
\[
 C_{\gamma_{\mathrm{re}}}^{\reg}(0;\,0;\,1)
 =
 -\frac12\log 2.
\]
This value uses the specified initial tangent vector; changing its scale changes the regularized value.
\end{eg}


\begin{eg}
Again take \(N=4\) and \(\gamma_{\mathrm{re}}:0\rightsquigarrow\pi/4\). Since
\[
 \eta_2^{\cot}=-\tan\theta\,d\theta,
 \qquad
 \eta_0^{\cot}=\cot\theta\,d\theta,
\]
our convention for the order of integration gives
\[
\begin{aligned}
 C_{\gamma_{\mathrm{re}}}^{\reg}(0;\,2,0;\,1)
 &=
 \int_{0<t_1<t_2<\pi/4}
 (-\tan t_1)\cot t_2\,dt_1\,dt_2 
 =
 \int_0^{\pi/4}\log(\cos t)\cot t\,dt \\
 &=
 \frac14\int_0^{1/2}\frac{\log(1-u)}{u}\,du 
 =
 -\frac14\operatorname{Li}_2\left(\frac12\right)
 =
 -\frac{\pi^2}{48}+\frac18\log^2 2.
\end{aligned}
\]
Here we used \(u=\sin^2t\) and the classical dilogarithm
\(\operatorname{Li}_2(x)=\sum_{n\geq1}x^n/n^2\).
The iterated integral converges without regularization: although
\(\eta_0^{\cot}\) has a pole at \(0\), the inner integral
\(\log(\cos t)\) vanishes to order two there.
\end{eg}

\subsection{An algebraic coordinate over the real cyclotomic field}
\label{subsec:cotangent-algebraic-coordinate}
We next replace the angular coordinate \(\theta\) by an algebraic coordinate \(v\) in which
the cotangent differential forms become logarithmic differential forms defined over
\(F\).

Define
\begin{equation*}
\label{eq:delta-lambda}
 \delta_N:=\zeta_N-\zeta_N^{-1}
 =
 2\sqrt{-1}\sin\left(\frac{2\pi}{N}\right),
 \qquad
 \lambda_N:=-\delta_N^2
 =
 4\sin^2\left(\frac{2\pi}{N}\right).
\end{equation*}
Then
\[
 \overline{\delta_N}=-\delta_N,
 \qquad
 \lambda_N\in F,
 \qquad
 E=F(\delta_N).
\]
We have \(\lambda_N>0\).
Put
\[
 z=e^{2\sqrt{-1}\theta}
\]
and introduce the coordinate
\begin{equation*}
\label{eq:cotangent-coordinate}
 v
 :=
 \frac{z-1}{\delta_N(z+1)}
 =
 \frac{e^{2\sqrt{-1}\theta}-1}
 {\delta_N(e^{2\sqrt{-1}\theta}+1)}.
\end{equation*}
Equivalently,
\[
 z=\frac{1+\delta_Nv}{1-\delta_Nv}.
\]
Using
\(
 \frac{e^{2\sqrt{-1}\theta}-1}
      {e^{2\sqrt{-1}\theta}+1}
 =
 \sqrt{-1}\tan\theta,
\)
we obtain
\[
 v
 =
 \frac{\sqrt{-1}}{\delta_N}\tan\theta
 =
 \frac{\tan\theta}{2\sin(2\pi/N)}.
\]

For \(a\in\mathbb Z/N\mathbb Z\) such that
\(\zeta_N^a\neq-1\), put
\begin{equation*}
\label{eq:cotangent-points}
 v_a
 :=
 \frac{\zeta_N^a-1}
 {\delta_N(\zeta_N^a+1)}
 =
 \frac{\sqrt{-1}}{\delta_N}\tan\theta_a.
\end{equation*}
Complex conjugation gives
\(
 \overline{v_a}=v_a,
\)
and hence
\[
 v_a\in F.
\]
If \(N\) is even and \(a=N/2\), then \(\zeta_N^a=-1\); in this case we set
\[
 v_a:=\infty.
\]
Thus, for every \(a\in\mathbb Z/N\mathbb Z\),
\[
 v_a\in\mathbb P^1(F).
\]

For \(v_a\neq\infty\), define the rational function
\[
 f_a(v)
 :=
 \frac{1+\lambda_Nv^2}{(v-v_a)^2}
 \in F(v)^\times,
\]
and the logarithmic form
\begin{equation}
\label{eq:kappa-finite}
 \kappa_a
 :=
 d\log f_a(v)
 =
 \left(
 \frac{2\lambda_Nv}{1+\lambda_Nv^2}
 -
 \frac{2}{v-v_a}
 \right)dv.
\end{equation}
If \(v_a=\infty\), define instead
\begin{equation*}
\label{eq:kappa-infinity}
 f_a(v):=1+\lambda_Nv^2,
 \qquad
 \kappa_a:=d\log(1+\lambda_Nv^2).
\end{equation*}
In particular, the expression in \eqref{eq:kappa-finite} is used only
when \(v_a\neq\infty\).

Let
\begin{equation}\label{eq:theta-to-v-map}
     \Phi:(\Theta, \theta)\longrightarrow\mathbb (\mathbb{P}^1_{\mathbb{C}},v),
 \qquad
 v=\Phi(\theta)
 :=
 \frac{e^{2\sqrt{-1}\theta}-1}
 {\delta_N(e^{2\sqrt{-1}\theta}+1)},
\end{equation}
where the value is understood to be \(\infty\) when the denominator vanishes. This map satisfies \(\Phi(\theta_a)=v_a\) for every \(a\in\mathbb Z/N\mathbb Z\).

\begin{prop}[Algebraic realization of the cotangent differential form]
\label{thm:algebraic-realization-cotangent-kernel}
For every \(a\in\mathbb Z/N\mathbb Z\), one has
\[
 \Phi^*\kappa_a=-2\,\eta_a^{\cot}
\]
on \(\Theta_N^\circ\).
Let
\(
 \gamma:\theta_{a_0}\rightsquigarrow\theta_{b_0}
\)
be a path equipped with tangential base points, and put
\[
 \gamma_v:=\Phi\circ\gamma:
 v_{a_0}\rightsquigarrow v_{b_0}.
\]
At the endpoints, \(\gamma_v\) is equipped with the tangential base points
obtained by applying \(d\Phi\) to those of \(\gamma\).  Then
\[
 \operatorname{reg}\int_{\gamma_v}
 \kappa_{a_1}\cdots\kappa_{a_r}
 =
 (-2)^r
 C_\gamma^{\reg}
 (a_0;\,a_1,\ldots,a_r;\,b_0).
\]
\end{prop}

\begin{proof}
Suppose first that \(v_a\neq\infty\).  From
\[
 v=\frac{\sqrt{-1}}{\delta_N}\tan\theta,
 \qquad
 v_a=\frac{\sqrt{-1}}{\delta_N}\tan\theta_a,
 \qquad
 \lambda_N=-\delta_N^2,
\]
we obtain
\[
 1+\lambda_Nv^2
 =
 1+\tan^2\theta
 =
 \frac{1}{\cos^2\theta}
\]
and
\[
 v-v_a
 =
 \frac{\sqrt{-1}}{\delta_N}
 \bigl(\tan\theta-\tan\theta_a\bigr).
\]
Since
\[
 \tan\theta-\tan\theta_a
 =
 \frac{\sin(\theta-\theta_a)}
 {\cos\theta\cos\theta_a},
\]
it follows that
\[
 \frac{1+\lambda_Nv^2}{(v-v_a)^2}
 =
 \lambda_N
 \frac{\cos^2\theta_a}
 {\sin^2(\theta-\theta_a)}.
\]
The factor \(\lambda_N\cos^2\theta_a\) is independent of \(\theta\).
Therefore
\[
\begin{aligned}
 \kappa_a
 =
 d\log
 \frac{1+\lambda_Nv^2}{(v-v_a)^2} 
 =
 -2\,d\log\sin(\theta-\theta_a) 
 =
 -2\cot(\theta-\theta_a)\,d\theta
 =
 -2\,\eta_a^{\cot}.
\end{aligned}
\]

Now suppose that \(v_a=\infty\).  Then \(N\) is even and
\(\theta_a=\pi/2\pmod{\pi}\).  Since
\[
 1+\lambda_Nv^2
 =
 1+\tan^2\theta
 =
 \sec^2\theta,
\]
we have
\[
 \kappa_a
 =
 d\log(\sec^2\theta)
 =
 2\tan\theta\,d\theta.
\]
On the other hand,
\[
 \cot\left(\theta-\frac{\pi}{2}\right)
 =
 -\tan\theta.
\]
Hence
\[
 \kappa_a
 =
 -2\cot(\theta-\theta_a)\,d\theta
 =
 -2\,\eta_a^{\cot}
\]
also in this case.

The identity for iterated integrals follows from the pullback formula,
multilinearity, and compatibility of tangential regularization with the
change of variables.
\end{proof}

\subsection{The cotangent punctured line}
\label{subsec:cotangent-punctured-line}
The preceding coordinate identifies the natural geometric domain of the
cotangent differential forms: a punctured projective line over \(F\) with one
nonsplit quadratic boundary divisor.

The polynomial
\[
 1+\lambda_Nv^2
\]
defines a reduced degree-two divisor on \(\mathbb P^1_F\).  We denote this
divisor by
\[
 D_N^{\mathrm{Cay}}
 :=
 V(1+\lambda_Nv^2).
\]
It is important that \(D_N^{\mathrm{Cay}}\) is defined over \(F\), although
its two geometric points are not individually \(F\)-rational.  Indeed, after
base change from \(F\) to \(E\), one has
\[
 (D_N^{\mathrm{Cay}})_E
 =
 \{\delta_N^{-1},-\delta_N^{-1}\}.
\]
Thus \(D_N^{\mathrm{Cay}}\) is a closed point of degree two over \(F\),
isomorphic to \(\operatorname{Spec}E\).

Define the reduced divisor
\[
 B_N^{\cot}
 :=
 \sum_{a\in\mathbb Z/N\mathbb Z}[v_a]
 +
 D_N^{\mathrm{Cay}}
\]
on \(\mathbb P^1_F\).

\begin{defn}
\label{def:cotangent-punctured-line}
The \textit{cotangent punctured line of level \(N\)} is the open curve
\[
 X_N^{\cot}
 :=
 \mathbb P^1_F\setminus\operatorname{Supp}(B_N^{\cot}).
\]
Its boundary consists of the \(F\)-rational cotangent points \(v_a\) and
the degree-two Cayley boundary divisor \(D_N^{\mathrm{Cay}}\).
\end{defn}

The map \(\Phi\) restricts to an analytic isomorphism
\(
 \Phi:\Theta_N^\circ\xrightarrow{\sim}X_N^{\cot}(\mathbb C)
\)
by \eqref{eq:theta-to-v-map}.

\begin{prop}
\label{prop:cotangent-forms-over-F}
For every \(a\in\mathbb Z/N\mathbb Z\), the form \(\kappa_a\) is an
\(F\)-rational logarithmic differential form on \(X_N^{\cot}\).
More precisely, if
\[
 f_a(v)=
 \begin{cases}
 \dfrac{1+\lambda_Nv^2}{(v-v_a)^2},
   & v_a\neq\infty,\\[8pt]
 1+\lambda_Nv^2,
   & v_a=\infty,
 \end{cases}
\]
then
\[
 \kappa_a=d\log f_a(v)
\]
and
\[
 \operatorname{div}(f_a)
 =
 D_N^{\mathrm{Cay}}-2[v_a].
\]
Thus \(\kappa_a\) has residue \(1\) at each of the two geometric Cayley
boundary points and residue \(-2\) at \(v_a\).
\end{prop}

\begin{proof}
Both \(v_a\) and \(\lambda_N\) belong to \(F\), so
\(
 f_a(v)\in F(v)^\times.
\)
Its zero divisor is \(D_N^{\mathrm{Cay}}\), while its pole divisor is
\(2[v_a]\), including the case \(v_a=\infty\).  Hence
\[
 \operatorname{div}(f_a)
 =
 D_N^{\mathrm{Cay}}-2[v_a].
\]
The asserted properties follow by taking the logarithmic differential.
\end{proof}

Proposition~\ref{thm:algebraic-realization-cotangent-kernel} shows that
iterated cotangent integrals are, up to the factor \((-2)^r\), iterated
integrals of the \(F\)-rational logarithmic forms
\[
 \kappa_a\in
 H^0\left(
  X_N^{\cot},
  \Omega^1_{\mathbb P^1_F}(\log B_N^{\cot})
 \right).
\]
In this sense, iterated cotangent integrals are periods naturally associated
with the cotangent punctured line.

\begin{rem}
The fact that \(X_N^{\cot}\) and the forms \(\kappa_a\) are defined over \(F\) does not, by itself, imply that every period of \(X_N^{\cot}\) is a mixed Tate period over \(F\).  The nonsplit divisor \(D_N^{\mathrm{Cay}}\) may contribute an Artin--Tate direction. However it splits into its two Cayley boundary points after extending scalars from \(F\) to \(E\). The descent construction studied below select particular families of periods for which a mixed Tate interpretation over a suitable ring of integers of \(F\) can be established. 
\end{rem}

\section{Cyclotomic realization}
\label{sec:cyclotomic-realization}
In this section, we will use the Cayley transformation
\[
 z=\frac{1+\delta_Nv}{1-\delta_Nv}
\]
to identify the resulting punctured line with the standard cyclotomic
punctured projective line and to express iterated cotangent integrals
explicitly in terms of cyclotomic iterated integrals. In fact,
after base change from \(F\) to the cyclotomic field \(E\), the cotangent punctured line becomes the standard cyclotomic punctured line. We will use this identification to express iterated cotangent integrals as cyclotomic iterated integrals and, for specified reference paths and tangential base points, as cyclotomic multiple zeta values.

\subsection{The Cayley transformation}
\label{subsec:cyclotomic-cayley-transformation}

We first identify the cotangent punctured line over \(E\) and compute the images of its logarithmic forms.

Retain the notation of the preceding section, and put
\[
 \xi_a:=\zeta_N^a,
 \qquad
 X_{N,E}^{\mathrm{cyc}}
 :=
 \mathbb P^1_E\setminus
 \{0,\mu_N,\infty\}.
\]
On the \(z\)-line, write
\[
 \omega_0:=\frac{dz}{z},
 \qquad
 \omega_\xi:=\frac{dz}{z-\xi}
 \quad(\xi\in\mu_N),
 \qquad
 h_a:=\omega_{\xi_a}-\frac12\omega_0.
\]

\begin{prop}[Cyclotomic realization of the cotangent punctured line]
\label{prop:cotangent-cyclotomic-isomorphism}
The Cayley transformation
\[
 \Psi(v):=\frac{1+\delta_Nv}{1-\delta_Nv},
 \qquad
 \Psi^{-1}(z)=\frac{z-1}{\delta_N(z+1)},
\]
induces an isomorphism
\[
 \Psi:(X_N^{\cot})_E
 \xrightarrow{\sim}X_{N,E}^{\mathrm{cyc}}.
\]
For every \(a\in\mathbb Z/N\mathbb Z\), one has
\[
 \kappa_a
 =
 \Psi^*(\omega_0-2\omega_{\xi_a})
 =
 -2\Psi^*h_a.
\]
\end{prop}

\begin{proof}
The transformation \(\Psi\) sends
\[
 v_a\longmapsto\xi_a,
 \qquad
 -\delta_N^{-1}\longmapsto0,
 \qquad
 \delta_N^{-1}\longmapsto\infty.
\]
When \(v_a=\infty\), its image is \(-1=\xi_a\). Thus \(\Psi\) identifies the boundary divisors and induces the stated isomorphism.

Suppose first that \(v_a\neq\infty\). Since \(\lambda_N=-\delta_N^2\), substitution gives
\[
 1+\lambda_Nv^2=\frac{4z}{(z+1)^2},
 \qquad
 v-v_a
 =
 \frac{2(z-\xi_a)}
 {\delta_N(z+1)(\xi_a+1)}.
\]
Consequently,
\[
 \frac{1+\lambda_Nv^2}{(v-v_a)^2}
 =
 \delta_N^2(\xi_a+1)^2
 \frac{z}{(z-\xi_a)^2}.
\]
Taking logarithmic differentials yields
\[
 \kappa_a
 =
 \Psi^*\bigl(d\log z-2d\log(z-\xi_a)\bigr).
\]
If \(v_a=\infty\), then \(\xi_a=-1\), and
\[
 \kappa_a
 =
 d\log(1+\lambda_Nv^2)
 =
 \Psi^*\left(d\log\frac{4z}{(z+1)^2}\right).
\]
This gives the same formula.
\end{proof}

\begin{rem}
The composition of \(\Phi\) from \eqref{eq:theta-to-v-map} with \(\Psi\) is
\[
 (\Psi\circ\Phi)(\theta)=e^{2\sqrt{-1}\theta}.
\]
It follows that
\[
 (\Psi\circ\Phi)^*h_a=\eta_a^{\cot}.
\]
Thus the analytic cotangent differential form corresponds to the cyclotomic logarithmic form \(h_a\).
\end{rem}

\subsection{Expansion into cyclotomic iterated integrals}
\label{subsec:cotangent-cyclotomic-expansion}
The preceding identity gives an explicit expansion along any admissible path, provided that its endpoint tangent vectors are transported by the coordinate changes.

Let
\[
 \gamma:\theta_{a_0}\rightsquigarrow\theta_{b_0}
\]
be a tangential path, and denote its images by
\[
 \gamma_v:=\Phi\circ\gamma:
 v_{a_0}\rightsquigarrow v_{b_0},
 \qquad
 \gamma_z:=\Psi\circ\gamma_v:
 \xi_{a_0}\rightsquigarrow\xi_{b_0}.
\]
The tangential base points of these image paths are obtained by applying the differentials of the corresponding maps.

For a tangential path \(\rho\) on \(X_{N,E}^{\mathrm{cyc}}(\mathbb C)\), write
\[
 I_\rho(\nu_1,\ldots,\nu_r)
 :=
 \operatorname{reg}\int_\rho
 \omega_{\nu_1}\cdots\omega_{\nu_r},
 \qquad
 \nu_j\in\{0\}\cup\mu_N.
\]
Our ordering convention is that the first letter corresponds to the earliest integration variable. For \(a\in\mathbb Z/N\mathbb Z\) and \(\varepsilon\in\{0,1\}\), put
\[
 \xi(a,\varepsilon)
 :=
 \begin{cases}
  0,&\varepsilon=0,\\
  \xi_a,&\varepsilon=1.
 \end{cases}
\]

\begin{prop}[Cyclotomic expansion of iterated cotangent integrals]
\label{prop:iterated-cotangent-cyclotomic-expansion}
With the paths and tangential data specified above, one has
\[
\begin{aligned}
 C_\gamma^{\reg}(a_0;\,a_1,\ldots,a_r;\,b_0)
 ={}&
 \sum_{\varepsilon_1,\ldots,\varepsilon_r\in\{0,1\}}
 \left(-\frac12\right)^{\#\{j:\varepsilon_j=0\}}
 I_{\gamma_z}\bigl(
 \xi(a_1,\varepsilon_1),\ldots,\xi(a_r,\varepsilon_r)
 \bigr).
\end{aligned}
\]
Equivalently,
\[
\begin{aligned}
 \operatorname{reg}\int_{\gamma_v}
 \kappa_{a_1}\cdots\kappa_{a_r}
 =
 \sum_{\varepsilon_1,\ldots,\varepsilon_r\in\{0,1\}}
 (-2)^{\#\{j:\varepsilon_j=1\}}
 I_{\gamma_z}\bigl(
 \xi(a_1,\varepsilon_1),\ldots,\xi(a_r,\varepsilon_r)
 \bigr).
\end{aligned}
\]
Every iterated integral on the right has length \(r\).
\end{prop}

\begin{proof}
By Proposition~\ref{prop:cotangent-cyclotomic-isomorphism} and compatibility of tangential regularization with pullback,
\[
\begin{aligned}
 C_\gamma^{\reg}(a_0;\,a_1,\ldots,a_r;\,b_0)
 =
 \operatorname{reg}\int_{\gamma_z}
 h_{a_1}\cdots h_{a_r}
 =
 \operatorname{reg}\int_{\gamma_z}
 \left(\omega_{\xi_{a_1}}-\frac12\omega_0\right)
 \cdots
 \left(\omega_{\xi_{a_r}}-\frac12\omega_0\right).
\end{aligned}
\]
Multilinearity gives the first formula. The second follows by expanding the corresponding product of the forms \(\omega_0-2\omega_{\xi_{a_j}}\).
\end{proof}

\subsection{Cyclotomic multiple zeta values}
\label{subsec:cotangent-cyclotomic-mzvs}
In this subsection, we introduce the cyclotomic multiple zeta values
arising from iterated integrals on the cyclotomic punctured line. 
The multiple
polylogarithms below are ordinary multivalued functions; the adjective
``cyclotomic'' refers to their regularized special values at roots of
unity.

For positive integers \(k_1,\ldots,k_s\), define 
the \textit{shuffle-type multiple polylogarithm} by
\[
\operatorname{Li}^{\shuffle}_{k_1,\ldots,k_s}
 (x_1,\ldots,x_s)
 :=
 \sum_{0<n_1<\cdots<n_s}
 \frac{
  x_1^{n_1}
  x_2^{n_2-n_1}
  \cdots
  x_s^{n_s-n_{s-1}}
 }{
  n_1^{k_1}\cdots n_s^{k_s}
 }.
\]
When $s=1$, we simply denote by $\operatorname{Li}_{k_1} (x_1)$.
The series is initially considered in its domain of absolute convergence.
Its values outside this domain are defined by analytic continuation.
At singular cyclotomic points, we use shuffle regularization with respect
to the standard tangential base points.

A regularized value
\[
 \operatorname{Li}^{\shuffle}_{k_1,\ldots,k_s}
 (\varepsilon_1,\ldots,\varepsilon_s),
 \qquad
 \varepsilon_1,\ldots,\varepsilon_s\in\mu_N,
\]
is called a \textit{cyclotomic multiple zeta value of level \(N\)} and
weight
\[
 k_1+\cdots+k_s.
\]
Let \(\mathcal Z_{N,r}\) denote the \(\mathbb Q\)-linear span
of regularized level-\(N\) cyclotomic multiple zeta values
of weight \(r\).

For each \(c \in \mathbb Z/N\mathbb Z\), we put
$\xi_c:=\zeta_N^c$
and let
\[
 \ell_c:(0,1]\longrightarrow\mathbb P^1_z,
 \qquad
 \ell_c(t)=t\xi_c,
\]
be the \textit{radial path} from the tangential base point at \(0\) in the direction
\(\xi_c\partial_z\) to the tangential base point at \(\xi_c\) with tangent
vector \(-\xi_c\partial_z\).

Recall that
\[
 h_a:=\omega_{\xi_a}-\frac12\omega_0,
 \qquad
 \omega_\xi:=\frac{dz}{z-\xi},
 \qquad
 \omega_0:=\frac{dz}{z}.
\]
\begin{defn}
For a sequence \(a_1,\ldots,a_r \in \mathbb Z/N \mathbb Z\), define
the \textit{radial \(h\)-integral} by
\begin{equation}\label{eq: radial h-integral}
     \mathcal H_c(a_1,\ldots,a_r)
 :=
 \operatorname{reg}\int_{\ell_c}
 h_{a_1}\cdots h_{a_r},
 \qquad
 \mathcal H_c(\varnothing):=1.
\end{equation}
\end{defn}

Under the change of variables
\(
 z=\xi_c t,
\)
the path \(\ell_c\) becomes the standard path from \(0\) to \(1\).
For \(\nu_1,\ldots,\nu_s\in\mu_N\), one has
\[
\begin{aligned}
&\operatorname{reg}\int_{\ell_c}
 \omega_{\nu_1}\omega_0^{k_1-1}
 \cdots
 \omega_{\nu_s}\omega_0^{k_s-1}
=
(-1)^s
\operatorname{Li}^{\shuffle}_{k_1,\ldots,k_s}
 \left(
  \frac{\xi_c}{\nu_1},
  \ldots,
  \frac{\xi_c}{\nu_s}
 \right).
\end{aligned}
\]
Here a power of a differential form denotes a repeated letter in the
iterated integral.

Expanding
\[
 h_{a_j}=\omega_{\xi_{a_j}}-\frac12\omega_0
\]
and applying shuffle regularization to words beginning with
\(\omega_0\), we obtain
\[
 \mathcal H_c(a_1,\ldots,a_r)\in\mathcal Z_{N,r}.
\]

\begin{eg}
We let
\[
 x_j:=\frac{\xi_c}{\xi_{a_j}}.
\]
Then
\[
 \mathcal H_c(a_1)
 =
 -\operatorname{Li}_1(x_1),
\]
and
\[
\begin{aligned}
 \mathcal H_c(a_1,a_2)
 ={}&
 \operatorname{Li}^{\shuffle}_{1,1}(x_1,x_2)
 +\frac12\operatorname{Li}_2(x_1)
 -\frac12\operatorname{Li}_2(x_2).
\end{aligned}
\]
In length three, one has
\[
\begin{aligned}
 \mathcal H_c(a_1,a_2,a_3)
& =
 -\operatorname{Li}^{\shuffle}_{1,1,1}(x_1,x_2,x_3)
+\frac12\operatorname{Li}^{\shuffle}_{2,1}(x_2,x_3)
 +\frac12\operatorname{Li}^{\shuffle}_{1,2}(x_2,x_3)
\\
\ &-\frac12\operatorname{Li}^{\shuffle}_{2,1}(x_1,x_3)
 -\frac12\operatorname{Li}^{\shuffle}_{1,2}(x_1,x_2)
-\frac14\operatorname{Li}_3(x_1)
 +\frac12\operatorname{Li}_3(x_2)
 -\frac14\operatorname{Li}_3(x_3).
\end{aligned}
\]
Thus every radial \(h\)-integral is an explicit level-\(N\)
cyclotomic multiple zeta value.  These radial values will be used below
as building blocks for integrals between two cyclotomic punctures.
\end{eg}
\subsection{Reference paths between cyclotomic points}
\label{subsec:cotangent-reference-paths}

We construct reference paths between cyclotomic points and express the
corresponding iterated cotangent integrals in terms of radial
\(h\)-integrals. Integer indices will record the choices of logarithmic
lifts, while the cotangent integrals  labels remain in \(\mathbb Z/N\mathbb Z\).

Recall that \(N\geq3\) and we write \(\xi_m:=\zeta_N^m\) for \(m\in\mathbb Z\); thus \(\xi_m\) depends only on \(m\bmod N\).
Words in iterated integrals are ordered by increasing path parameter,
and path products are traversed from left to right.

\subsubsection{Reference paths and integer indices}

We use integer indices to specify tangential rotations without
introducing a separate notation for each integer lift.

For \(m\in\mathbb Z\), let
\[
 \ell_m(t):=t\xi_m,
 \qquad 0<t<1,
\]
be the radial path with tangential endpoints
\[
 \ell_m:
 (0,\xi_m\partial_z)
 \rightsquigarrow
 (\xi_m,-\xi_m\partial_z).
\]
Define the tangential rotation
\[
 \tau_m(s)
 :=
 \exp\left(\frac{2\pi\sqrt{-1}\,ms}{N}\right)\partial_z,
 \qquad 0\leq s\leq1,
\]
and put
\[
 p_m:=\tau_m*\ell_m:
 (0,\partial_z)
 \rightsquigarrow
 (\xi_m,-\xi_m\partial_z).
\]
Although \(\ell_m\) depends only on \(m\bmod N\), the homotopy
classes of \(\tau_m\) and \(p_m\) depend on the integer \(m\).
Replacing \(m\) by \(m+Nk\) adds \(k\) full turns to the
tangential rotation.

Let \(a_0,b_0\in\mathbb Z/N\mathbb Z\), and choose integers \(m,n\)
such that
\[
 m\bmod N=a_0,
 \qquad
 n\bmod N=b_0.
\]
Define
\[
 \rho_{m,n}:=p_m^{-1}*p_n:
 (\xi_{a_0},-\xi_{a_0}\partial_z)
 \rightsquigarrow
 (\xi_{b_0},-\xi_{b_0}\partial_z).
\]
This composition is taken in the tangential fundamental groupoid of
\[
 \mathbb P^1(\mathbb C)\setminus\{0,\mu_N,\infty\}.
\]
It can be represented by travelling inward along the first radial
path, following a sufficiently small arc around \(0\) with the
prescribed winding, and then travelling outward along the second
radial path. Such a representative does not pass through \(0\).

Let \(\gamma_{m,n}\) be the corresponding path in the angular
cylinder under
\[
 e(\theta):=e^{2\sqrt{-1}\theta}.
\]
Its lift to \(\mathbb C\) starting at \(\pi m/N\) ends at
\(\pi n/N\). At both endpoints, the transported tangent vector is
\[
 -\frac{1}{2\sqrt{-1}}\partial_\theta
 =\frac{\sqrt{-1}}{2}\partial_\theta,
\]
because its image under \(de\) is \(-\xi_c\partial_z\) at
\(\theta_c\).

For a word \(a_1,\ldots,a_r\in\mathbb Z/N\mathbb Z\), write
\[
 \mathcal H_m(a_1,\ldots,a_r)
 :=
 \operatorname{reg}\int_{\ell_m}h_{a_1}\cdots h_{a_r},
 \qquad
 \mathcal H_m(\varnothing):=1.
\]
This is the radial \(h\)-integral of
\eqref{eq: radial h-integral}, indexed here by an integer \(m\);
in particular, \(\mathcal H_{m+N}=\mathcal H_m\).
Define the \textit{framed radial \(h\)-integral} by
\[
 \mathcal H_m^{\mathrm{fr}}(a_1,\ldots,a_r)
 :=
 \operatorname{reg}\int_{p_m}h_{a_1}\cdots h_{a_r},
 \qquad
 \mathcal H_m^{\mathrm{fr}}(\varnothing):=1.
\]

Every \(h_a\) has residue \(-1/2\) at \(0\). On the punctured
tangent line, it is therefore represented by
\(-\frac12\,dq/q\), where \(q\) is the tangent coordinate relative
to \(\partial_z\). Hence
\[
 \operatorname{reg}\int_{\tau_m}h_{a_1}\cdots h_{a_j}
 =
 \frac{1}{j!}
 \left(-\frac{\pi\sqrt{-1}\,m}{N}\right)^j.
\]
Chen's path-composition formula gives
\begin{equation}
\label{eq:framed-radial-h-integral}
 \mathcal H_m^{\mathrm{fr}}(a_1,\ldots,a_r)
 =
 \sum_{j=0}^{r}
 \frac{1}{j!}
 \left(-\frac{\pi\sqrt{-1}\,m}{N}\right)^j
 \mathcal H_m(a_{j+1},\ldots,a_r).
\end{equation}
Throughout this subsection, every empty word has value \(1\).

\subsubsection{The endpoint formula}
Path composition and reversal express the integral between two
cyclotomic endpoints as a finite sum of products of framed radial
\(h\)-integrals.

\begin{thm}[Cyclotomic endpoint formula]
\label{thm:cyclotomic-endpoint-formula}
Let \(m,n\in\mathbb Z\) represent \(a_0,b_0\in\mathbb Z/N\mathbb Z\),
respectively. For the reference path \(\gamma_{m,n}\), with the
tangential data specified above, one has
\[
\begin{aligned}
 &C_{\gamma_{m,n}}^{\reg}
 (a_0;\,a_1,\ldots,a_r;\,b_0)
 =
 \sum_{j=0}^{r}(-1)^j
 \mathcal H_m^{\mathrm{fr}}(a_j,\ldots,a_1)
 \mathcal H_n^{\mathrm{fr}}(a_{j+1},\ldots,a_r).
\end{aligned}
\]
In particular,
\[
 C_{\gamma_{m,n}}^{\reg}
 (a_0;\,a_1,\ldots,a_r;\,b_0)
 \in\mathcal Z_{N,r}.
\]
\end{thm}

\begin{proof}
The pullback identity \(e^*h_a=\eta_a^{\cot}\), together with the
transported tangential data, gives
\[
 C_{\gamma_{m,n}}^{\reg}
 (a_0;\,a_1,\ldots,a_r;\,b_0)
 =
 \operatorname{reg}\int_{\rho_{m,n}}h_{a_1}\cdots h_{a_r}.
\]
Apply Chen's composition formula to
\(\rho_{m,n}=p_m^{-1}*p_n\). Path reversal yields
\[
 \operatorname{reg}\int_{p_m^{-1}}h_{a_1}\cdots h_{a_j}
 =
 (-1)^j\mathcal H_m^{\mathrm{fr}}(a_j,\ldots,a_1),
\]
which proves the endpoint formula.

By the preceding subsection,
\[
 \mathcal H_m(a_1,\ldots,a_s)\in\mathcal Z_{N,s}.
\]
Moreover, for the standard branches,
\[
 \operatorname{Li}^{\shuffle}_1(\zeta_N)
 -
 \operatorname{Li}^{\shuffle}_1(\zeta_N^{-1})
 =
 \frac{N-2}{N}\pi\sqrt{-1}.
\]
Since \(N\geq3\), this implies
\(\pi\sqrt{-1}\in\mathcal Z_{N,1}\).
Equation~\eqref{eq:framed-radial-h-integral} and the shuffle product
therefore give
\[
 \mathcal H_m^{\mathrm{fr}}(a_1,\ldots,a_s)
 \in\mathcal Z_{N,s}.
\]
Each product in the endpoint formula has total weight \(r\).
A further application of the shuffle product proves the assertion.
\end{proof}

For explicit calculations, it is useful to combine the two tangential
rotations. Put
\[
 \Delta_{m,n}:=\frac{\pi\sqrt{-1}}{N}(m-n).
\]
Since
\[
 \rho_{m,n}
 =
 \ell_m^{-1}*\tau_m^{-1}*\tau_n*\ell_n,
\]
Chen's formula also gives
\begin{equation}
\label{eq:cotangent-endpoint-radial-formula}
\begin{aligned}
 &C_{\gamma_{m,n}}^{\reg}
 (a_0;\,a_1,\ldots,a_r;\,b_0)
 =
 \sum_{\substack{j,q\geq0\\j+q\leq r}}
 (-1)^j\frac{\Delta_{m,n}^{\,q}}{q!}
 \mathcal H_m(a_j,\ldots,a_1)
 \mathcal H_n(a_{j+q+1},\ldots,a_r).
\end{aligned}
\end{equation}
Here the \(q\) letters assigned to the middle tangential path all
contribute the same residue \(-1/2\).

\subsubsection{Explicit endpoint formulas in low weights}
We now combine the cyclotomic endpoint formula with the explicit
expressions for the radial \(h\)-integrals obtained in
Subsection~\ref{subsec:cotangent-cyclotomic-mzvs}.

Let \(m,n\in\mathbb Z\) represent
\(a_0,b_0\in\mathbb Z/N\mathbb Z\), respectively, and set
\[
 x_j:=\frac{\xi_m}{\xi_{a_j}},
 \qquad
 y_j:=\frac{\xi_n}{\xi_{a_j}},
 \qquad
 \Delta_{m,n}:=\frac{\pi\sqrt{-1}}{N}(m-n).
\]
We abbreviate \(\Delta_{m,n}\) to \(\Delta\) in the following formulas. Since
\[
 \operatorname{Li}_1(\zeta_N)
 -
 \operatorname{Li}_1(\zeta_N^{-1})
 =
 \frac{N-2}{N}\pi\sqrt{-1},
\]
the winding contribution can itself be written as the weight-one
cyclotomic multiple zeta value
\[
 \Delta
 =
 \frac{m-n}{N-2}
 \left(
  \operatorname{Li}_1(\zeta_N)
  -
  \operatorname{Li}_1(\zeta_N^{-1})
 \right).
\]
All singular values occurring below are understood by shuffle
regularization with respect to the tangential conventions fixed above.

\begin{eg}[Length one]
The endpoint formula becomes
\[
 C_{\gamma_{m,n}}^{\reg}(a_0;\,a_1;\,b_0)
 =
 \operatorname{Li}^{\shuffle}_1(x_1)
 -
 \operatorname{Li}^{\shuffle}_1(y_1)
 +\Delta.
\]
Thus both the endpoint logarithms and the prescribed winding
contribute to the value.
\end{eg}

\begin{eg}[Length two]
One has
\[
\begin{aligned}
 &C_{\gamma_{m,n}}^{\reg}(a_0;\,a_1,a_2;\,b_0)
 =
 \operatorname{Li}^{\shuffle}_{1,1}(y_1,y_2)
 +
 \operatorname{Li}^{\shuffle}_{1,1}(x_2,x_1)
 -
 \operatorname{Li}^{\shuffle}_1(x_1)
 \operatorname{Li}^{\shuffle}_1(y_2)\\
\  &+\frac12\Bigl(
 \operatorname{Li}^{\shuffle}_2(y_1)
 -
 \operatorname{Li}^{\shuffle}_2(y_2)
 +
 \operatorname{Li}^{\shuffle}_2(x_2)
 -
 \operatorname{Li}^{\shuffle}_2(x_1)
 \Bigr)
 +\Delta\Bigl(
 \operatorname{Li}^{\shuffle}_1(x_1)
 -
 \operatorname{Li}^{\shuffle}_1(y_2)
 \Bigr)
 +\frac{\Delta^2}{2}.
\end{aligned}
\]
Products can be rewritten as linear combinations of CMZVs by shuffle
identities. For instance,
\[
 \operatorname{Li}^{\shuffle}_1(x_1)
 \operatorname{Li}^{\shuffle}_1(y_2)
 =
 \operatorname{Li}^{\shuffle}_{1,1}(x_1,y_2)
 +
 \operatorname{Li}^{\shuffle}_{1,1}(y_2,x_1).
\]
\end{eg}

\begin{eg}[Length three]
Substituting the length-one, length-two, and length-three radial formulas
from Subsection~\ref{subsec:cotangent-cyclotomic-mzvs} into
\eqref{eq:cotangent-endpoint-radial-formula}, we obtain
\begin{align*}
C_{\gamma_{m,n}}^{\reg}&(a_0;\,a_1,a_2,a_3;\,b_0)
={}-\Li^{\shuffle}_{1,1,1}(y_1,y_2,y_3)
    +\Li^{\shuffle}_{1,1,1}(x_3,x_2,x_1)\\
&+\tfrac12\bigl\{
 \Li^{\shuffle}_{2,1}(y_2,y_3)+\Li^{\shuffle}_{1,2}(y_2,y_3)
-\Li^{\shuffle}_{2,1}(y_1,y_3)-\Li^{\shuffle}_{1,2}(y_1,y_2)\bigr\}\\
&+\tfrac12\bigl\{
-\Li^{\shuffle}_{2,1}(x_2,x_1)-\Li^{\shuffle}_{1,2}(x_2,x_1)
+\Li^{\shuffle}_{2,1}(x_3,x_1)+\Li^{\shuffle}_{1,2}(x_3,x_2)\bigr\}\\
&-\tfrac14\Li_3(y_1)+\tfrac12\Li_3(y_2)-\tfrac14\Li_3(y_3)
 +\tfrac14\Li_3(x_3)-\tfrac12\Li_3(x_2)+\tfrac14\Li_3(x_1)\\
&+\Li_1(x_1)\bigl\{
 \Li^{\shuffle}_{1,1}(y_2,y_3)
 +\tfrac12\Li_2(y_2)-\tfrac12\Li_2(y_3)\bigr\}\\
&-\bigl\{
 \Li^{\shuffle}_{1,1}(x_2,x_1)
 +\tfrac12\Li_2(x_2)-\tfrac12\Li_2(x_1)\bigr\}\Li_1(y_3)\\
&+\Delta\bigl\{
 \Li^{\shuffle}_{1,1}(y_2,y_3)
 +\Li^{\shuffle}_{1,1}(x_2,x_1)
 -\Li_1(x_1)\Li_1(y_3)\\
& 
 +\tfrac12\Li_2(y_2)-\tfrac12\Li_2(y_3)
 +\tfrac12\Li_2(x_2)-\tfrac12\Li_2(x_1)\bigr\}\\
&+\tfrac{\Delta^2}{2}
 \bigl\{\Li_1(x_1)-\Li_1(y_3)\bigr\}
 +\tfrac{\Delta^3}{6}.
\end{align*}
Every term in this expression has total weight three.  Products have
been retained to keep the formula readable.  Replacing \(\Delta\) by
its cyclotomic logarithmic expression above and applying the shuffle
product gives a \(\mathbb Q\)-linear combination of weight-three
level-\(N\) cyclotomic multiple zeta values.
\end{eg}

\begin{rem}
The integers \(m\) and \(n\) record the chosen lifts of the endpoint
labels and hence are part of the path data.  Simultaneously replacing
\((m,n)\) by \((m+Nk,n+Nk)\) does not change the tangential homotopy
class of \(\rho_{m,n}\), whereas changing only one of them changes its
winding around \(0\).  The resulting formulas therefore depend, for
fixed endpoints, only on \(m-n\).  Winding around the other punctures
requires additional path-composition terms.
\end{rem}

\begin{rem}
The cyclotomic formulas above are obtained after base change from
\(F=E^+\) to \(E=\mathbb Q(\mu_N)\).  They prove that the corresponding
periods belong to \(\mathcal Z_{N,r}\), but do not by themselves give
a mixed Tate interpretation over \(F\), since the transported
tangential data need not be \(F\)-rational.  The descent constructions
in the following sections address this point.
\end{rem}

\section{Cotangent-difference forms and the difference punctured line}
\label{sec:cotangent-difference-integrals}

Note that taking differences of cotangent differential forms removes their common contribution at the quadratic boundary of the punctured cotangent line. We will use this cancellation to realize the resulting integrals on a projective line punctured only at \(F\)-rational points, and then establish their mixed Tate interpretation.

Recall that \(N\geq3\),\(E=\mathbb Q(\mu_N)\), and \(F=E^+\). We retain the angular coordinates \(\theta_a\), the algebraic points \(v_a\), and the forms \(\eta_a^{\cot}\) and \(\kappa_a\) from the preceding sections.

\subsection{Cotangent-difference forms and their iterated integrals}
\label{subsec:cotangent-difference-kernels}
We define the cotangent-difference forms and derive their expressions
in the algebraic coordinate $v$, including the case in which one of
their poles lies at infinity.

For \(a,b\in\mathbb Z/N\mathbb Z\), put
\[
 \eta_{a,b}^{\mathrm{diff}}
 :=
 \eta_a^{\cot}-\eta_b^{\cot}
 =
 \bigl(
  \cot(\theta-\theta_a)-\cot(\theta-\theta_b)
 \bigr)\,d\theta.
\]
These forms satisfy
\[
 \eta_{a,a}^{\mathrm{diff}}=0,
 \qquad
 \eta_{a,b}^{\mathrm{diff}}
 =-\eta_{b,a}^{\mathrm{diff}},
 \qquad
 \eta_{a,b}^{\mathrm{diff}}
 +\eta_{b,c}^{\mathrm{diff}}
 =\eta_{a,c}^{\mathrm{diff}}.
\]

To describe the forms algebraically, use homogeneous coordinates
\([X_0:X_1]\) on \(\mathbb P^1_F\), with \(v=X_0/X_1\), and define
\[
 L_a(X_0,X_1)
 :=
 \begin{cases}
  X_0-v_aX_1, & v_a\neq\infty,\\
  X_1,     & v_a=\infty.
 \end{cases}
\]
The ratio \(L_a/L_b\) is a rational function over \(F\). Put
\[
 \kappa^{\mathrm{diff}}_{a,b}:=d\log\frac{L_a}{L_b}.
\]
Multiplying either linear form by a nonzero constant does not change
\(\kappa^{\mathrm{diff}}_{a,b}\). In the affine coordinate \(v\), this gives
\[
 \kappa^{\mathrm{diff}}_{a,b}
 =
 \begin{cases}
  \displaystyle\frac{dv}{v-v_a}-\frac{dv}{v-v_b},
     & v_a,v_b\neq\infty,\\[6pt]
  \displaystyle\frac{dv}{v-v_a},
     & v_a\neq\infty,\ v_b=\infty,\\[6pt]
  \displaystyle-\frac{dv}{v-v_b},
     & v_a=\infty,\ v_b\neq\infty,\\[6pt]
  0, & v_a=v_b.
 \end{cases}
\]

\begin{prop}[Algebraic realization of the cotangent-difference forms]
\label{prop:cotangent-difference-kernels}
Let \(\Phi\) denote the angular-to-algebraic coordinate map defined
in (\ref{eq:theta-to-v-map}). Then
\[
 \kappa^{\mathrm{diff}}_{a,b}=-\frac12(\kappa_a-\kappa_b),
 \qquad
 \Phi^*\kappa^{\mathrm{diff}}_{a,b}=\eta_{a,b}^{\mathrm{diff}}.
\]
If \(a\neq b\), the only poles of \(\kappa^{\mathrm{diff}}_{a,b}\) are \(v_a\)
and \(v_b\), with residues \(1\) and \(-1\), respectively.
In particular, \(\kappa^{\mathrm{diff}}_{a,b}\) is regular at both geometric
points of the quadratic boundary divisor.
\end{prop}

\begin{proof}
In homogeneous coordinates, the definition of \(\kappa_a\) can
be written uniformly as
\[
 \kappa_a
 =
 d\log\frac{X_1^2-\delta_N^2X_0^2}{L_a(X_0,X_1)^2}.
\]
This includes \(v_a=\infty\), since then \(L_a=X_1\).
Consequently,
\[
 \kappa_a-\kappa_b
 =
 -2d\log\frac{L_a}{L_b}
 =
 -2\kappa^{\mathrm{diff}}_{a,b}.
\]
Applying \(\Phi^*\kappa_a=-2\eta_a^{\cot}\) proves the pullback
identity. The statements about poles and residues follow from
\[
 \operatorname{div}\left(\frac{L_a}{L_b}\right)
 =
 [v_a]-[v_b].
\]
\end{proof}

\begin{defn}[Iterated cotangent-difference integral]
\label{def:cotangent-difference-integral}
Let
\[
 a_0,b_0,a_1,b_1,\ldots,a_r,b_r\in\mathbb Z/N\mathbb Z,
\]
and let
\(\gamma:\theta_{a_0}\rightsquigarrow\theta_{b_0}\)
be an admissible path with prescribed tangential endpoint data.
\textit{Iterated cotangent-difference integral} is defined by
\[
 C_\gamma^{\mathrm{diff},\reg}
 \left(
  a_0;\,(a_1,b_1),\ldots,(a_r,b_r);\,b_0
 \right)
 :=
 \operatorname{reg}\int_\gamma
 \eta_{a_1,b_1}^{\mathrm{diff}}\cdots
 \eta_{a_r,b_r}^{\mathrm{diff}}.
\]
Its weight is the length \(r\). The empty integral has value \(1\).
\end{defn}

If \(\gamma_v=\Phi(\gamma)\), with the transported tangential
data, then Proposition~\ref{prop:cotangent-difference-kernels} gives
\begin{equation}
\label{eq:cotangent-difference-algebraic-integral}
 C_\gamma^{\mathrm{diff},\reg}
 \left(
  a_0;\,(a_1,b_1),\ldots,(a_r,b_r);\,b_0
 \right)
 =
 \operatorname{reg}\int_{\gamma_v}
 \kappa^{\mathrm{diff}}_{a_1,b_1}\cdots\kappa^{\mathrm{diff}}_{a_r,b_r}.
\end{equation}
Equivalently, by multilinearity,
\[
\begin{aligned}
 &C_\gamma^{\mathrm{diff},\reg}
 \left(
  a_0;\,(a_1,b_1),\ldots,(a_r,b_r);\,b_0
 \right)
 =
 \sum_{\varepsilon_1,\ldots,\varepsilon_r\in\{0,1\}}
 (-1)^{\varepsilon_1+\cdots+\varepsilon_r}
 C_\gamma^{\reg}
 \left(
  a_0;\,c_{1,\varepsilon_1},\ldots,c_{r,\varepsilon_r};\,b_0
 \right),
\end{aligned}
\]
where \(c_{j,0}=a_j\) and \(c_{j,1}=b_j\).

\begin{eg}[Length one]
Assume first that the four points
\[
 v_{a_0},\quad v_{b_0},\quad v_{a_1},\quad v_{b_1}
\]
are finite and that the integral converges without regularization. Since
\[
 \eta_{a_1,b_1}^{\mathrm{diff}}
 =
 d\log
 \frac{\sin(\theta-\theta_{a_1})}
      {\sin(\theta-\theta_{b_1})},
\]
we obtain
\[
\begin{aligned}
&C_\gamma^{\mathrm{diff}}
 \bigl(a_0;\,(a_1,b_1);\,b_0\bigr)
=
 \log_\gamma
 \frac{
  \sin(\theta_{b_0}-\theta_{a_1})
  \sin(\theta_{a_0}-\theta_{b_1})
 }{
  \sin(\theta_{b_0}-\theta_{b_1})
  \sin(\theta_{a_0}-\theta_{a_1})
 }.
\end{aligned}
\]
Equivalently, in the \(v\)-coordinate,
\[
\begin{aligned}
&C_\gamma^{\mathrm{diff}}
 \bigl(a_0;\,(a_1,b_1);\,b_0\bigr)
=
 \log_{\gamma_v}
 \frac{
  (v_{b_0}-v_{a_1})(v_{a_0}-v_{b_1})
 }{
  (v_{b_0}-v_{b_1})(v_{a_0}-v_{a_1})
 }.
\end{aligned}
\]
Thus a length-one cotangent-difference integral is a logarithm of an
\(F\)-rational cross-ratio. If some of the points coincide or are at
infinity, the same formula is interpreted using homogeneous coordinates
and tangential regularization.
\end{eg}

\begin{eg}[Length two]
By multilinearity,
\[
\begin{aligned}
C_\gamma^{\mathrm{diff},\reg}
 \bigl(a_0;\,(a_1,b_1),(a_2,b_2);\,b_0\bigr)
=
&
 C_\gamma^{\reg}(a_0;\,a_1,a_2;\,b_0)
 -C_\gamma^{\reg}(a_0;\,a_1,b_2;\,b_0)\\
&-C_\gamma^{\reg}(a_0;\,b_1,a_2;\,b_0)
 +C_\gamma^{\reg}(a_0;\,b_1,b_2;\,b_0).
\end{aligned}
\]
In particular, the cotangent-difference integrals are explicit
alternating sums of the iterated cotangent integrals introduced earlier.
\end{eg}


\begin{eg}[A level-\(4\) example]
Let \(N=4\), and let \(\gamma\) be the increasing real interval from
\(\theta_0=0\) to \(\theta_1=\pi/4\). Neither
\(\eta_2^{\cot}\) nor \(\eta_3^{\cot}\) has a pole on this interval.
The length-one formula gives
\[
\begin{aligned}
 C_\gamma^{\mathrm{diff}}
 \bigl(0;\,(2,3);\,1\bigr)
 =
 \log
 \frac{
  \sin(\pi/4-\pi/2)\sin(0-3\pi/4)
 }{
  \sin(\pi/4-3\pi/4)\sin(0-\pi/2)
 }
 =-\log 2.
\end{aligned}
\]
Consequently,
\[
 C_\gamma^{\mathrm{diff}}
 \bigl(0;\,(2,3),(2,3);\,1\bigr)
 =
 \frac12\log^2 2.
\]
\end{eg}

\subsection{The difference punctured line}
\label{subsec:cotangent-difference-punctured-line}
In this subsection, we identify the algebraic curve naturally underlying
the cotangent-difference forms.  We show that these forms extend across
the quadratic boundary of the cotangent punctured line, leaving a
projective line punctured only at \(F\)-rational points.

Define the reduced divisor
\[
 D_N^{\mathrm{diff}}
 :=
 \sum_{a\in\mathbb Z/N\mathbb Z}[v_a]
\]
on \(\mathbb P^1_F\), and define the \textit{difference punctured line} by
\[
 X_N^{\mathrm{diff}}
 :=
 \mathbb P^1_F\setminus D_N^{\mathrm{diff}}.
\]
The points \(v_a\) are distinct and belong to \(\mathbb P^1(F)\).
Since the cotangent punctured line is obtained by removing, in addition, the
nonsplit quadratic boundary divisor
\( D_\infty:= \{\delta_N^{-1},-\delta_N^{-1}\}\),
there is a natural open immersion
\[
 j:X_N^{\cot}\hookrightarrow X_N^{\mathrm{diff}}.
\]
Every \(\kappa^{\mathrm{diff}}_{a,b}\) extends to a logarithmic form on
\(X_N^{\mathrm{diff}}\), so
\eqref{eq:cotangent-difference-algebraic-integral} depends only on
the image of the path in this larger curve.

For a fixed \(c\in\mathbb Z/N\mathbb Z\), the forms
\[
 \kappa^{\mathrm{diff}}_{a,c},
 \qquad a\neq c,
\]
form an \(F\)-basis of
\(H^1_{\mathrm{dR}}(X_N^{\mathrm{diff}}/F)\).
Indeed, the residue map identifies this space with
\[
 \left\{
  (r_a)_{a\in\mathbb Z/N\mathbb Z}\in F^N:
  \sum_a r_a=0
 \right\},
\]
and \(\kappa^{\mathrm{diff}}_{a,c}\) corresponds to the vector with entries
\(1\) at \(a\), \(-1\) at \(c\), and zero elsewhere.
At the motivic level, the localization sequence gives
\[
 H^1(X_N^{\mathrm{diff}})
 \simeq
 \mathbb Q(-1)^{\oplus(N-1)}.
\]

After base change to \(E\), the Cayley transformation identifies
\[
 (X_N^{\mathrm{diff}})_E
 \simeq
 \mathbb P^1_E\setminus\mu_N,
\]
and under this identification,
\[
 \kappa^{\mathrm{diff}}_{a,b}
 =
 \omega_{\xi_a}-\omega_{\xi_b}.
\]
Thus the cyclotomic points \(z=0\) and \(z=\infty\) are ordinary
points of this curve, rather than punctures.

We next specify the integral base. Let \(S\) be the set of finite
places of \(F\) above the rational primes dividing \(N\), and put
\[
 \mathcal O_{F,S}=\mathcal O_F[1/N].
\]

\begin{prop}[Good reduction of the marked punctured line]
\label{prop:cotangent-difference-good-reduction}
The marked curve
\[
 \bigl(\mathbb P^1_F,(v_a)_{a\in\mathbb Z/N\mathbb Z}\bigr)
\]
admits a model over \(\mathcal O_{F,S}\) whose marked sections
are pairwise disjoint.
\end{prop}

\begin{proof}
Choose three distinct labels \(c_0,c_1,c_\infty\), and let
\[
 t:\mathbb P^1_F\xrightarrow{\sim}\mathbb P^1_F
\]
be the unique projective coordinate satisfying
\[
 t(v_{c_0})=0,
 \qquad
 t(v_{c_1})=1,
 \qquad
 t(v_{c_\infty})=\infty.
\]
Write \(t_a=t(v_a)\). By invariance of cross-ratios under the
Cayley transformation, for \(a\neq c_\infty\),
\[
 t_a
 =
 \frac{
  (\xi_a-\xi_{c_0})(\xi_{c_1}-\xi_{c_\infty})
 }{
  (\xi_a-\xi_{c_\infty})(\xi_{c_1}-\xi_{c_0})
 }.
\]
Every nonzero difference of two \(N\)-th roots of unity is a unit
after inverting \(N\). Hence
\[
 t_a\in\mathcal O_{F,S}^{\times}
 \qquad
 (a\neq c_0,c_\infty).
\]
Similarly, for distinct \(a,b\neq c_\infty\),
\[
 t_a-t_b
 =
 \frac{
  (\xi_{c_1}-\xi_{c_\infty})
  (\xi_{c_0}-\xi_{c_\infty})(\xi_a-\xi_b)
 }{
  (\xi_{c_1}-\xi_{c_0})
  (\xi_a-\xi_{c_\infty})(\xi_b-\xi_{c_\infty})
 }
 \in\mathcal O_{F,S}^{\times}.
\]
These expressions lie in \(F\), since \(t\) and the points \(v_a\)
are defined over \(F\). Their unit property over the cyclotomic
extension therefore implies the asserted unit property over \(F\).

Consequently, the finite points \(t_a\) define pairwise disjoint
sections of \(\mathbb A^1_{\mathcal O_{F,S}}\), also disjoint
from the section at infinity. They give the required model.
\end{proof}

We fix such a normalized coordinate \(t\) and write
\[
 \mathcal X_N^{\mathrm{diff}}
 :=
 \mathbb P^1_{\mathcal O_{F,S}}
 \setminus\bigcup_a[t_a]
\]
for the resulting smooth integral model.
At a finite marked point we may use \(t-t_a\) as a local parameter;
at infinity we use \(1/t\). A tangential vector is called
\textit{\(S\)-integral} here if its value under the differential of
the corresponding local parameter is an element of
\(\mathcal O_{F,S}^{\times}\).
In particular, these choices provide nonzero \(S\)-integral
tangential vectors at every puncture.

\section{Motivic cotangent-difference periods and Goncharov's coproduct}
\label{sec:cotangent-difference-motivic-coaction}

In this section, we give a motivic interpretation of
cotangent-difference integrals over the totally real cyclotomic field \(F=E^+\). The difference forms extend to logarithmic forms on \(X_N^{\mathrm{diff}}\), a projective line over \(F\) punctured only at \(F\)-rational points. Base change to \(E\) recovers their cyclotomic description.  With \(F\)-rational tangential endpoints, the integrals are mixed-Tate periods over \(F\); with \(S\)-integral endpoint data, they are periods over \(\mathcal O_{F,S}\). We then show that the algebra generated by the corresponding graded framed-motive classes is stable under Goncharov's coproduct.

\subsection{Motivic periods and mixed-Tate descent}
\label{subsec:motdiff-motivic-lift}

We first recall the motivic period and framed-object conventions, and then construct the motivic lifts of cotangent-difference integrals.

Put \(\mathcal T:=\operatorname{MT}(\mathcal O_{F,S})\), and fix the
embedding \(\sigma:F\hookrightarrow\mathbb C\) induced by
\(E\subset\mathbb C\). The category of mixed-Tate motives over rings
of \(S\)-integers and its motivic fundamental groupoids are constructed
in \cite[Sections~1 and~4]{DeligneGoncharov2005}.


For \(M\in\mathcal T\), let \(\omega_{\mathrm{dR}}(M)\) and \(\omega_{\mathrm B,\sigma}(M)\) denote its de Rham and Betti realizations. The comparison isomorphism associated with \(\sigma\) is
\[
 \operatorname{comp}_\sigma:
 \omega_{\mathrm{dR}}(M)\otimes_{F,\sigma}\mathbb C
 \xrightarrow{\sim}
 \omega_{\mathrm B,\sigma}(M)\otimes_{\mathbb Q}\mathbb C.
\]
A motivic period is represented by a triple
\[
 [M,f,v]^{\mathfrak m},\qquad
 v\in\omega_{\mathrm{dR}}(M),\quad
 f\in\omega_{\mathrm B,\sigma}(M)^\vee.
\]
Functoriality and linearity identify equivalent triples, and tensor products define their multiplication. Their numerical values are obtained by the comparison pairing
\[
 \operatorname{per}_\sigma[M,f,v]^{\mathfrak m}
 =f\bigl(\operatorname{comp}_\sigma(v)\bigr).
\]
Write \(\mathcal P^{\mathfrak m}(\mathcal T)\) for the algebra of these motivic matrix coefficients and
\[
 \operatorname{Per}_\sigma(\mathcal T)
 :=\operatorname{per}_\sigma\bigl(\mathcal P^{\mathfrak m}(\mathcal T)\bigr)
 \subset\mathbb C
\]
for its numerical period algebra. Thus a formal class \([M,f,v]^{\mathfrak m}\) retains its motivic data, while \(\operatorname{per}_\sigma\) returns its complex comparison value. We use the same conventions for \(\operatorname{MT}(F)\); see \cite[Sections~2.2--2.3]{BrownMotivicPeriods}.

Let $\mathcal A^{\mathcal M}_{F,S}$ be the graded Hopf algebra
of framed objects in $\mathcal T$. An $r$-framed object is an
object $P\in\mathcal T$ equipped with nonzero morphisms
\[
 v_0:\mathbb Q(0)\longrightarrow\operatorname{gr}_0^W P,
 \qquad
 f_r:\operatorname{gr}_{-2r}^W P\longrightarrow\mathbb Q(r).
\]
The $r$-th degree component $\mathcal A^{\mathcal M}_{F,S,r}$ consists of classes of $r$-framed objects under the equivalence relation generated by morphisms respecting both frames. The direct sum with the induced frames defines the addition in $\mathcal A^{\mathcal M}_{F,S}$, and the tensor product defines the multiplication in $\mathcal A^{\mathcal M}_{F,S}$. The coproduct is defined by inserting dual bases of intermediate graded pieces $\operatorname{gr}_{-2p}^W P$, for $0\leq p\leq r$; see
\cite[Section~8.2]{Goncharov2005}. We denote it by
$\Delta_{\mathcal M}$.

Let \(\gamma:\theta_{a_0}\rightsquigarrow\theta_{b_0}\) be an admissible angular path, and let \(\gamma_v=\Phi(\gamma)\) be its image on \(X_N^{\mathrm{diff}}(\mathbb C)\). Assume that the transported endpoint tangents \(\vec x\) and \(\vec y\) are nonzero and \(F\)-rational. Put
\[
 \lambda_j:=\kappa^{\mathrm{diff}}_{a_j,b_j},
 \qquad \boldsymbol\lambda=(\lambda_1,\ldots,\lambda_r).
\]
The Deligne--Goncharov construction attaches a unipotent motivic path torsor \(\Pi^{\mathfrak m}_{\vec x,\vec y}\) to these base points. Its coordinate ring has finite mixed-Tate subobjects
\[
 M_{\leq r}(\vec x,\vec y)
 :=L_{\leq r}\mathcal O(\Pi^{\mathfrak m}_{\vec x,\vec y}),
\]
where \(L\) is the bar-length filtration; see \cite[Section~4]{DeligneGoncharov2005}.

The bar word \(v_{\boldsymbol\lambda}=[\lambda_1|\cdots|\lambda_r]\) defines a de Rham vector in \(M_{\leq r}(\vec x,\vec y)\), and evaluation at the Betti path defines a covector \(f_{\gamma_v}(g)=g([\gamma_v]_{\mathrm B})\). Every word in these logarithmic forms is integrable: the forms are closed and their pairwise wedge products vanish on the curve. In the convention that \([M,f,v]^{\mathfrak m}\) has a Betti covector \(f\) and a de Rham vector \(v\), set
\begin{equation}
 I_{\gamma_v}^{\mathfrak m}
 (\vec x;\lambda_1|\cdots|\lambda_r;\vec y)
 :=[M_{\leq r}(\vec x,\vec y),f_{\gamma_v},v_{\boldsymbol\lambda}]^{\mathfrak m}.
 \label{eq:motdiff-motivic-lift}
\end{equation}
Tangential regularization throughout uses local parameters normalized by the specified tangent vectors, as in the analytic definition.

\begin{thm}[Mixed-Tate descent of cotangent-difference integrals]
\label{thm:cotangent-difference-mixed-tate}
With the endpoint data above,
\begin{equation}\label{eq: comparison motivic period}
 I_{\gamma_v}^{\mathfrak m}
 (\vec x;\lambda_1|\cdots|\lambda_r;\vec y)
 \in\mathcal P^{\mathfrak m}(\operatorname{MT}(F)),
\end{equation}
and its comparison period is
\begin{equation}
\begin{aligned}
 \operatorname{per}_\sigma I_{\gamma_v}^{\mathfrak m}
 (\vec x;\lambda_1|\cdots|\lambda_r;\vec y)
 &=\operatorname{reg}\int_{\gamma_v}\lambda_1\cdots\lambda_r\\
 &=C_\gamma^{\mathrm{diff},\mathrm{reg}}
 \bigl(a_0;(a_1,b_1),\ldots,(a_r,b_r);b_0\bigr).
\end{aligned}
\label{eq:motdiff-period-comparison}
\end{equation}
If \(\vec x\) and \(\vec y\) generate the corresponding tangent lines over \(\mathcal O_{F,S}\), then this motivic period belongs to
\[
 \mathcal P^{\mathfrak m}\bigl(\operatorname{MT}(\mathcal O_{F,S})\bigr).
\]
\end{thm}

\begin{proof}
The curve \(X_N^{\mathrm{diff}}\) is a projective line with \(F\)-rational punctures, and the forms \(\lambda_j\) are defined over \(F\). Its motivic path torsors between \(F\)-rational tangential base points are therefore ind-mixed-Tate objects, by \cite[Section~4]{DeligneGoncharov2005}. The coefficient in \eqref{eq:motdiff-motivic-lift} lies in the finite object \(M_{\leq r}(\vec x,\vec y)\), so it is a mixed-Tate motivic period.

The comparison pairing evaluates the de Rham word against \(\gamma_v\), giving the first equality in \eqref{eq:motdiff-period-comparison}. The second follows from \(\Phi^*\kappa^{\mathrm{diff}}_{a,b}=\eta^{\mathrm{diff}}_{a,b}\), with transported tangential data.

By Proposition~\ref{prop:cotangent-difference-good-reduction}, the curve has a smooth proper compactification over \(\mathcal O_{F,S}\) with disjoint rational boundary sections. When the endpoint tangents generate their integral tangent lines, the integral path-torsor construction applies; see \cite[Proposition~4.17 and Remark~4.18]{DeligneGoncharov2005}.
\end{proof}

\begin{rem}
The rationality condition concerns the algebraic tangent vectors, not the Betti path; the numerical period need not be real. Transported angular tangents need not be \(F\)-rational. Rescaling an admissible tangent by an \(S\)-unit preserves the integral condition above.
\end{rem}
\subsection{The Hopf algebra of motivic cotangent differences}
\label{subsec:motdiff-coaction-stability}

We express cotangent-difference forms in a fixed logarithmic basis
and apply Goncharov's coproduct to the corresponding Tate-framed
motivic iterated integrals.

Retain the coordinate \(t\) and the label \(c_\infty\) chosen above,
so that \(t(v_{c_\infty})=\infty\), and put
\[
 J=(\mathbb Z/N\mathbb Z)\setminus\{c_\infty\}.
\]
For \(d\in J\), write \(t_d=t(v_d)\) and set
\begin{equation*}
 \kappa_d^\circ
 :=\kappa^{\mathrm{diff}}_{d,c_\infty}
 =\frac{dt}{t-t_d},
 \qquad
 \kappa_{c_\infty}^\circ:=0.
 \label{eq:motdiff-fixed-alphabet}
\end{equation*}
The second assignment is only a notational convention. We have
identities of differential forms
\begin{equation*}
 \kappa^{\mathrm{diff}}_{a,b}
 =\kappa_a^\circ-\kappa_b^\circ,
 \qquad
 H^1_{\mathrm{dR}}(X_N^{\mathrm{diff}}/F)
 =\bigoplus_{d\in J}F\kappa_d^\circ.
 \label{eq:motdiff-letter-expansion}
\end{equation*}
More explicitly, if
\[
 \varepsilon_{a,b}(d):=\delta_{a,d}-\delta_{b,d},
\]
then
\begin{equation}
 \kappa^{\mathrm{diff}}_{a,b}
 =\sum_{d\in J}\varepsilon_{a,b}(d)\kappa_d^\circ.
 \label{eq:motdiff-difference-basis-epsilon}
\end{equation}
Since each \(\kappa_d^\circ\) is itself a difference form, words in
difference forms and words in the \(\kappa_d^\circ\) have the same
\(\mathbb Q\)-span. The equality in
\eqref{eq:motdiff-difference-basis-epsilon} holds for forms, so it
may be applied to iterated integrals letter by letter.

After base change to \(E\), the Cayley coordinate gives
\[
 \kappa_d^\circ
 =\frac{dz}{z-\xi_d}
  -\frac{dz}{z-\xi_{c_\infty}}.
\]
Thus the two residues at the restored point \(z=\infty\) cancel.

We now fix the tangential base points used throughout the coproduct
formula. For \(a\in\mathbb Z/N\mathbb Z\), put
\[
 q_a=
 \begin{cases}
  t-t_a,&a\in J,\\
  1/t,&a=c_\infty,
 \end{cases}
 \qquad
 \vec v_a
 :=\left.\frac{\partial}{\partial q_a}\right|_{v_a},
\]
and define
\begin{equation}
 \mathcal B_N:=\{\vec v_a:a\in\mathbb Z/N\mathbb Z\}.
 \label{Definition of basepoint set}
\end{equation}
These tangent vectors are defined over \(F\) and generate the
tangent lines of the chosen integral model.

Goncharov's motivic fundamental groupoid is defined using fixed
tangential base points at the punctures; in an affine coordinate,
his standard choice satisfies \(dt(\vec v_a)=1\)
\cite[Section~3]{Goncharov2005}. We use that choice at the finite
punctures and the local coordinate \(q_{c_\infty}=1/t\) at infinity.
All intermediate base points appearing in the coproduct are taken
from this same system.

For \(\beta_0,\beta_1\in\mathcal B_N\) and \(d_1,\ldots,d_r\in J\),
let
\[
 I^{\mathfrak m}
 (\beta_0;\kappa_{d_1}^\circ|\cdots|
             \kappa_{d_r}^\circ;\beta_1)
 \in\mathcal A^{\mathcal M}_{F,S,r}
\]
denote Goncharov's Tate-framed motivic iterated-integral class
attached to this sequence of logarithmic forms and the specified
tangential base points. In particular, the framing and Tate
normalization are those of
\cite[Sections~3 and~8]{Goncharov2005}. We extend the notation
multilinearly to difference forms. Thus
\begin{equation}
\begin{aligned}
&I^{\mathfrak m}
 (\beta_0;
  \kappa^{\mathrm{diff}}_{a_1,b_1}|\cdots|
  \kappa^{\mathrm{diff}}_{a_r,b_r};
  \beta_1)
 =
 \sum_{d_1,\ldots,d_r\in J}
 \left(\prod_{j=1}^r\varepsilon_{a_j,b_j}(d_j)\right)
 I^{\mathfrak m}
 (\beta_0;\kappa_{d_1}^\circ|\cdots|
             \kappa_{d_r}^\circ;\beta_1).
\end{aligned}
\label{eq:motdiff-signed-expansion}
\end{equation}

The motivic period
\(I_{\gamma_v}^{\mathfrak m}\) defined in 
\eqref{eq: comparison motivic period}
also records a Betti path \(\gamma_v\). The class
\(I^{\mathfrak m}\) above is instead a Tate-framed class in the
graded Hopf algebra \(\mathcal A^{\mathcal M}_{F,S}\).
The coproduct below acts on the latter class.

Define the unital algebra
\begin{equation*}
\begin{aligned}
 \mathcal H_N^{\mathrm{diff},\mathfrak m}(\mathcal B_N)
 :=\mathbb Q\bigl[
 &I^{\mathfrak m}\bigl(
   \beta_0;
   \kappa^{\mathrm{diff}}_{a_1,b_1}|\cdots|
   \kappa^{\mathrm{diff}}_{a_r,b_r};
   \beta_1\bigr):
\ r\geq0,\
 \beta_0,\beta_1\in\mathcal B_N,\
 a_j,b_j\in\mathbb Z/N\mathbb Z
 \bigr]
 \subseteq\mathcal A^{\mathcal M}_{F,S}.
\end{aligned}
\label{eq:motdiff-motivic-algebra-definition}
\end{equation*}
For fixed endpoints, the generators satisfy the shuffle product.

\begin{thm}[Hopf algebra of motivic cotangent differences]
\label{thm:cotangent-difference-motivic-hopf}
Let \(d_1,\ldots,d_r\in J\) and
\(\beta_0,\beta_1\in\mathcal B_N\). For
\[
 0=i_0<i_1<\cdots<i_k<i_{k+1}=r+1,
\]
set
\[
 \beta^{(0)}=\beta_0,\qquad
 \beta^{(k+1)}=\beta_1,\qquad
 \beta^{(h)}=\vec v_{d_{i_h}}\quad(1\leq h\leq k).
\]
Then Goncharov's coproduct is
\begin{equation}
\begin{aligned}
&\Delta_{\mathcal M}
 I^{\mathfrak m}
 (\beta_0;\kappa_{d_1}^\circ|\cdots|
             \kappa_{d_r}^\circ;\beta_1)=\\
 \sum_{k=0}^{r}
 \sum_{0=i_0<\cdots<i_k<i_{k+1}=r+1}&
 I^{\mathfrak m}
 (\beta_0;\kappa_{d_{i_1}}^\circ|\cdots|
             \kappa_{d_{i_k}}^\circ;\beta_1)
 \otimes
 \prod_{h=0}^{k}
 I^{\mathfrak m}\bigl(
    \beta^{(h)};
    \kappa_{d_{i_h+1}}^\circ|\cdots|
    \kappa_{d_{i_{h+1}-1}}^\circ;
    \beta^{(h+1)}\bigr).
\end{aligned}
\label{eq:motdiff-goncharov-coproduct}
\end{equation}
Empty words have value \(1\). Consequently,
\begin{equation}
 \Delta_{\mathcal M}
 \mathcal H_N^{\mathrm{diff},\mathfrak m}(\mathcal B_N)
 \subseteq
 \mathcal H_N^{\mathrm{diff},\mathfrak m}(\mathcal B_N)
 \otimes_{\mathbb Q}
 \mathcal H_N^{\mathrm{diff},\mathfrak m}(\mathcal B_N).
 \label{eq:motdiff-coproduct-closure}
\end{equation}
Hence
\(\mathcal H_N^{\mathrm{diff},\mathfrak m}(\mathcal B_N)\)
is a connected graded Hopf subalgebra of
\(\mathcal A^{\mathcal M}_{F,S}\).
\end{thm}

\begin{proof}
For the Tate-framed classes just defined,
\eqref{eq:motdiff-goncharov-coproduct} is Goncharov's cut formula
\cite[Theorems~1.2 and~3.3]{Goncharov2005}.
The fixed tangential base points are used in every factor.
The assertion with an endpoint at infinity is understood in
the local coordinate \(q_{c_\infty}=1/t\), with its tangent
transported under the coordinate change.

Put
\(H=\mathcal H_N^{\mathrm{diff},\mathfrak m}(\mathcal B_N)\).
Each retained letter \(\kappa_d^\circ\) is a difference form.
Every complementary block is a word in the same forms, and its
endpoints belong to \(\mathcal B_N\). Thus every tensor factor on
the right-hand side of
\eqref{eq:motdiff-goncharov-coproduct} belongs to \(H\).
Multilinearity and \eqref{eq:motdiff-signed-expansion} give the
same conclusion for all generators of \(H\). Since the coproduct
is multiplicative, \eqref{eq:motdiff-coproduct-closure} follows.

Word length gives a nonnegative grading with \(H_0=\mathbb Q\).
The counit restricts to \(H\), so \(H\) is a connected graded
sub-bialgebra. 
For a positive-degree homogeneous element \(x\in H\), write
\[
 \Delta_{\mathcal M}(x)
 =x\otimes1+1\otimes x+\sum_\nu x'_\nu\otimes x''_\nu,
\]
where both factors in the reduced coproduct have smaller positive degrees. The ambient antipode satisfies
\[
 \mathsf S(x)=-x-\sum_\nu\mathsf S(x'_\nu)x''_\nu.
\]
Induction on degree, starting with \(\mathsf S(1)=1\), gives \(\mathsf S(H)\subseteq H\). This proves the Hopf subalgebra assertion.
\end{proof}

\begin{rem}[Coproduct for difference words]
\label{rem:motdiff-explicit-coproduct}
Let \(\lambda_j=\kappa^{\mathrm{diff}}_{a_j,b_j}\). For a cut
\(0=i_0<i_1<\cdots<i_k<i_{k+1}=r+1\) and
\(d_1,\ldots,d_k\in J\), put
\[
 \beta^{(0)}=\beta_0,\qquad
 \beta^{(k+1)}=\beta_1,\qquad
 \beta^{(h)}=\vec v_{d_h}\quad(1\leq h\leq k).
\]
Expanding the retained letters and recombining the omitted
letters gives
\begin{equation*}
\begin{aligned}
&\Delta_{\mathcal M}
 I^{\mathfrak m}
 (\beta_0;\lambda_1|\cdots|\lambda_r;\beta_1)\\
&=
 \sum_{k=0}^{r}
 \sum_{0=i_0<\cdots<i_k<i_{k+1}=r+1}
 \sum_{d_1,\ldots,d_k\in J}
 \left(
  \prod_{h=1}^{k}
  \varepsilon_{a_{i_h},b_{i_h}}(d_h)
 \right)\\
&\qquad\quad{}\times
 I^{\mathfrak m}
 (\beta_0;\kappa_{d_1}^\circ|\cdots|
             \kappa_{d_k}^\circ;\beta_1)
 \otimes
 \prod_{h=0}^{k}
 I^{\mathfrak m}\bigl(
    \beta^{(h)};
    \lambda_{i_h+1}|\cdots|\lambda_{i_{h+1}-1};
    \beta^{(h+1)}\bigr).
\end{aligned}
\label{eq:motdiff-explicit-difference-coproduct}
\end{equation*}
For \(k=0\), the sum over \(d_1,\ldots,d_k\) has one term.
Empty words have value \(1\).
\end{rem}

\begin{eg}[Length two]
With \(\lambda_j=\kappa^{\mathrm{diff}}_{a_j,b_j}\),
\begin{equation*}
\begin{aligned}
&\Delta_{\mathcal M}
 I^{\mathfrak m}
 (\beta_0;\lambda_1|\lambda_2;\beta_1)\\
&=
 I^{\mathfrak m}
 (\beta_0;\lambda_1|\lambda_2;\beta_1)\otimes1
 +1\otimes
 I^{\mathfrak m}
 (\beta_0;\lambda_1|\lambda_2;\beta_1)\\
&\quad+
 \sum_{d\in J}\varepsilon_{a_1,b_1}(d)
 I^{\mathfrak m}(\beta_0;\kappa_d^\circ;\beta_1)
 \otimes
 I^{\mathfrak m}(\vec v_d;\lambda_2;\beta_1)\\
&\quad+
 \sum_{d\in J}\varepsilon_{a_2,b_2}(d)
 I^{\mathfrak m}(\beta_0;\kappa_d^\circ;\beta_1)
 \otimes
 I^{\mathfrak m}(\beta_0;\lambda_1;\vec v_d).
\end{aligned}
\label{eq:motdiff-weight-two-example}
\end{equation*}
The cut points \(\vec v_d\) explain why
\(\mathcal B_N\) contains a fixed tangential base point at
every puncture.
\end{eg}

\section{Low-weight examples and real cyclotomic units}
\label{sec:low-weight-examples}
We illustrate the preceding constructions through explicit evaluations of
cotangent-difference integrals. In length one, these integrals give
logarithms of real cyclotomic \(S\)-units. At higher weights, examples
at level \(4\) recover classical zeta values and polylogarithmic constants.


\subsection{Length one and cyclotomic cross-ratios}
\label{subsec:length-one-cyclotomic-cross-ratios}
We first express length-one cotangent-difference integrals as logarithms
of cross-ratios of cyclotomic points.

Let \(p,q,a,b\in\mathbb Z/N\mathbb Z\) be pairwise distinct, and let
\(\gamma:\theta_p\rightsquigarrow\theta_q\) be an admissible path.
Since
\[
 \eta_{a,b}^{\mathrm{diff}}
 =
 d\log\frac{\sin(\theta-\theta_a)}
              {\sin(\theta-\theta_b)},
\]
we obtain
\begin{equation*}
\label{eq:length-one-difference-cross-ratio}
 C_\gamma^{\mathrm{diff},\reg}(p;\,(a,b);\,q)
 =
 \log_\gamma R(p,q;a,b),
\end{equation*}
where
\[
 R(p,q;a,b)
 :=
 \frac{
  (\xi_q-\xi_a)(\xi_p-\xi_b)
 }{
  (\xi_q-\xi_b)(\xi_p-\xi_a)
 }.
\]
Here \(\log_\gamma R(p,q;a,b)\) denotes the logarithm
along \(\gamma\), rather than an independently chosen branch of the
logarithm. No endpoint regularization is needed under the present
assumptions.

Indeed, the identity
\[
 \xi_c-\xi_d
 =
 2\sqrt{-1}\,
 e^{\sqrt{-1}(\theta_c+\theta_d)}
 \sin(\theta_c-\theta_d)
\]
gives
\[
 R(p,q;a,b)
 =
 \frac{
  \sin(\theta_q-\theta_a)\sin(\theta_p-\theta_b)
 }{
  \sin(\theta_q-\theta_b)\sin(\theta_p-\theta_a)
 }.
\]
The exponential factors cancel, and this expression is independent
of the chosen real lifts of the angular coordinates.

\begin{prop}
\label{prop:cyclotomic-cross-ratio-s-unit}
The cross-ratio \(R(p,q;a,b)\) belongs to
\(\mathcal O_{F,S}^{\times}\). More precisely, it is a real
cyclotomic \(S\)-unit.
\end{prop}

\begin{proof}
The cross-ratio belongs to \(E\), and its sine expression shows that
it is real. Hence it belongs to \(F=E^+\).

Every nonzero difference of \(N\)-th roots of unity is a unit after
inverting the primes above \(N\). Indeed,
\[
 \xi_c-\xi_d=\xi_d(\zeta_N^{c-d}-1),
\]
and the prime divisors of \(1-\zeta_N^{c-d}\) lie above primes
dividing \(N\). Thus \(R(p,q;a,b)\) is an \(S\)-unit in \(E\).
Since it belongs to \(F\), it lies in
\(\mathcal O_{F,S}^{\times}\). Its expression as a quotient of
cyclotomic differences gives the final assertion.
\end{proof}

For \(m\neq0\) in \(\mathbb Z/N\mathbb Z\), put
\[
 u_m
 :=
 (1-\zeta_N^m)(1-\zeta_N^{-m})
 =
 4\sin^2\left(\frac{\pi m}{N}\right).
\]
These are totally positive elements of
\(\mathcal O_{F,S}^{\times}\), and
\[
 R(p,q;a,b)^2
 =
 \frac{u_{q-a}u_{p-b}}{u_{q-b}u_{p-a}}.
\]
Consequently,
\begin{equation*}
\label{eq:real-part-length-one-difference}
 \operatorname{Re}
 C_\gamma^{\mathrm{diff},\reg}(p;\,(a,b);\,q)
 =
 \frac12\log
 \frac{u_{q-a}u_{p-b}}{u_{q-b}u_{p-a}}.
\end{equation*}

If \(\gamma\) is an increasing real interval without intermediate
punctures, then the integral is real and
\(R(p,q;a,b)>0\). In this case,
\[
 C_\gamma^{\mathrm{diff},\reg}(p;\,(a,b);\,q)
 =
 \log_\gamma R(p,q;a,b).
\]
For a general complex path, the value may acquire logarithmic
monodromy and need not be real.

\begin{rem}
A real cyclotomic \(S\)-unit need not be a unit in
\(\mathcal O_F\): primes above \(N\) are allowed in its denominator.
The proposition identifies a concrete family of such \(S\)-units;
it does not assert that these cross-ratios generate all of
\(\mathcal O_{F,S}^{\times}\). For the classical theory of circular
units, see Sinnott~\cite{Sinnott1978}.
\end{rem}

\subsection{A real cyclotomic unit at level five}
\label{subsec:level-five-cyclotomic-unit}
At level \(5\), the length-one construction recovers the logarithm
of a genuine unit in the real cyclotomic field.

Take \(N=5\), so that \(F=\mathbb Q(\sqrt5)\), and let
\(\gamma_5\) be the increasing real interval from \(0\) to
\(\pi/5\). The forms with labels \(2\) and \(3\) have no poles
on this interval. Therefore
\[
\begin{aligned}
 C_{\gamma_5}^{\mathrm{diff},\reg}(0;\,(2,3);\,1)
 &=
 \log\frac{\sin(\pi/5)}{\sin(2\pi/5)}
 =
 -\log\left(2\cos\frac{\pi}{5}\right)
 =
 -\log\frac{1+\sqrt5}{2}.
\end{aligned}
\]
Put
\[
 \varepsilon:=\frac{1+\sqrt5}{2}.
\]
Since \(\varepsilon^2-\varepsilon-1=0\), both \(\varepsilon\) and
\(\varepsilon^{-1}=\varepsilon-1\) are algebraic integers.
Thus \(\varepsilon\in\mathcal O_F^\times\). Moreover,
\[
 \varepsilon
 =
 -\zeta_5^2(1+\zeta_5)
 =
 -\zeta_5^2\frac{1-\zeta_5^2}{1-\zeta_5},
\]
which exhibits it as a real cyclotomic unit.

Repeated cotangent-difference forms give higher-weight examples immediately. By the
shuffle identity,
\[
 C_{\gamma_5}^{\mathrm{diff},\reg}
 \left(0;\,
 \underbrace{(2,3),\ldots,(2,3)}_{r};
 \,1\right)
 =
 \frac{(-\log\varepsilon)^r}{r!}.
\]

\subsection{Level four and classical polylogarithms}
\label{subsec:level-four-polylogarithms}
At level \(4\), the difference punctured line has a particularly
simple coordinate, making the first few weights explicitly computable.

Take \(N=4\) and set
\[
 t:=\tan\theta.
\]
The four cotangent points become
\[
 \theta_0\longmapsto0,\qquad
 \theta_1\longmapsto1,\qquad
 \theta_2\longmapsto\infty,\qquad
 \theta_3\longmapsto-1.
\]
Thus, in this coordinate,
\[
 X_4^{\mathrm{diff}}
 \simeq
 \mathbb P^1_{\mathbb Q}\setminus\{0,1,-1,\infty\}.
\]
The difference cotangent-difference forms satisfy
\begin{equation*}
\label{eq:level-four-difference-kernels}
 \eta_{0,2}^{\mathrm{diff}}=\frac{dt}{t},
 \qquad
 \eta_{1,2}^{\mathrm{diff}}=\frac{dt}{t-1},
 \qquad
 \eta_{3,2}^{\mathrm{diff}}=\frac{dt}{t+1}.
\end{equation*}

Let \(\gamma_4\) be the increasing angular interval
\(0\rightsquigarrow\pi/4\), whose image in the \(t\)-line is
\(0\rightsquigarrow1\). We use the tangential vectors
\(\partial_t\) at \(0\) and \(-\partial_t\) at \(1\), and transport
them to the angular coordinate.

\begin{eg}[Length one]
The nonsingular differential form gives
\[
 C_{\gamma_4}^{\mathrm{diff},\reg}(0;\,(3,2);\,1)
 =
 \int_0^1\frac{dt}{1+t}
 =
 \log2.
\]
For comparison, the singular differential forms have regularized values
\[
 C_{\gamma_4}^{\mathrm{diff},\reg}(0;\,(0,2);\,1)=0,
 \qquad
 C_{\gamma_4}^{\mathrm{diff},\reg}(0;\,(1,2);\,1)=0.
\]
These equalities use the specified tangential vectors. For example,
the constant term of
\(\int_{\epsilon}^{1}dt/t=-\log\epsilon\) is zero.
\end{eg}

\begin{prop}
\label{prop:level-four-single-zeta-values}
For every integer \(r\geq2\),
\[
 C_{\gamma_4}^{\mathrm{diff},\reg}
 \left(0;\,(1,2),
 \underbrace{(0,2),\ldots,(0,2)}_{r-1};
 \,1\right)
 =
 -\zeta(r),
\]
and
\[
 C_{\gamma_4}^{\mathrm{diff},\reg}
 \left(0;\,(3,2),
 \underbrace{(0,2),\ldots,(0,2)}_{r-1};
 \,1\right)
 =
 \left(1-2^{1-r}\right)\zeta(r).
\]
All these integrals converge without regularization.
\end{prop}

\begin{proof}
With upper endpoint \(T<1\), successive integration and
\[
 \frac{d}{dT}\operatorname{Li}_r(\pm T)
 =
 \frac{\operatorname{Li}_{r-1}(\pm T)}{T}
\]
identify the two iterated integrals with
\(-\operatorname{Li}_r(T)\) and
\(-\operatorname{Li}_r(-T)\), respectively.
For \(r\geq2\), their limits at \(T=1\) exist, and
\[
 \operatorname{Li}_r(1)=\zeta(r),
 \qquad
 \operatorname{Li}_r(-1)
 =
 -\left(1-2^{1-r}\right)\zeta(r).
\]
Absolute convergence follows from the same calculation: on
\(0<t<1\), each constituent differential form has constant sign, so the
integral of the absolute value is the absolute value of the
corresponding convergent integral.
\end{proof}

These examples display both features of the construction. The
length-one values are logarithms of real cyclotomic units or
\(S\)-units, while higher lengths produce genuinely iterated
polylogarithmic periods. In the level-\(4\) coordinate, the punctures
remain disjoint over \(\mathbb Z[1/2]\), and the chosen tangent
vectors generate the corresponding integral tangent lines.
Theorem~\ref{thm:cotangent-difference-mixed-tate} therefore places
these examples
\(
 \operatorname{Per}\bigl(\operatorname{MT}(\mathbb Z[1/2])\bigr).
\)
Their cyclotomic expressions and their mixed Tate interpretation
thus describe the same integrals from complementary viewpoints.

\bigskip
{\bf Acknowledgements.}
The work of J.~C.~ was partially supported by the following grants: International Exchange and Development Program for Young Faculty at USTB (Grant No. QNXM20250045), the University of Science and Technology Beijing Foundation (Grant No. 00007886) and the Fundamental Research Funds for the Central Universities (Grant No. 06320202). The work of H.~F.~ was partially supported by the following grants: JSPS KAKENHI JP24K00520, JP24K21510. 
The authors are grateful to Henrik Bachmann and Ryota Umezawa for informing them about the related references.

\medskip
{\bf Declaration of Generative AI Use.}
ChatGPT (OpenAI), a generative AI tool, was used during the
preparation of this manuscript to assist with the exploration and
exposition of selected mathematical arguments. In particular, it
suggested studying iterated integrals of differences of cotangent
kernels. All mathematical statements, calculations, and proofs
included in the final manuscript were critically reviewed and
independently verified by the authors, who take full responsibility
for its content.


\end{document}